\documentclass[11pt]{article}
\usepackage{amsmath,amssymb,fullpage,color,url}
\usepackage[tmargin=1.0in,bmargin=1.0in,rmargin=1.0in,lmargin=1.0in]{geometry}
\usepackage[breaklinks=true]{hyperref}
\usepackage{graphicx}
\usepackage{verbatim}
\newtheorem{lemma}[equation]{Lemma}
\newtheorem{theorem}[equation]{Theorem}
\newtheorem{defn}[equation]{Definition}

\newtheorem{prop}[equation]{Proposition}

        {\medskip}

\newenvironment{proof}{\vspace{-0.05in}\noindent{\bf Proof.}}%
        {\hspace*{\fill}$\Box$\par}
        {\hspace*{\fill}$\Box$\par\vspace{4mm}}
        {\hspace*{\fill}$\Box$\par}

\makeatletter
\newcommand{\mypm}{\mathbin{\mathpalette\@mypm\relax}}
\newcommand{\@mypm}[2]{\ooalign{%
		\raisebox{.1\height}{$#1+$}\cr
		\smash{\raisebox{-.6\height}{$#1-$}}\cr}}
\makeatother
\usepackage{blindtext}
\usepackage{textcomp}
\usepackage{commath}
\usepackage{mathtools}
\usepackage{amsmath}

\usepackage{hyperref}
\numberwithin{equation}{section}

\title{Vortex-type equations on compact Riemann surfaces}
\author{Kartick Ghosh}
\date{}
\begin{document} 	 	
	 	\maketitle
	\begin{abstract} In this paper, we prove \emph{a priori} estimates  for some vortex-type equations on compact Riemann surfaces. As applications, we recover existing estimates for the vortex bundle Monge-Amp\`ere equation, prove an existence and uniqueness theorem for the Calabi-Yang-Mills equations on vortex bundles and get estimates for $J-$vortex equation. We prove an existence and uniqueness result relating Gieseker stability and the existence of almost Hermitian Einstein metrics, i.e., a Kobayashi-Hitchin type correspondence. We also prove K\"ahlerness of the negative of the symplectic form which arises in the moment map interpretation of the Calabi-Yang-Mills equations in \cite{Vamsi3}. 
	\end{abstract}
	\section{Introduction.}\label{sec:intro}
		Let $\Sigma$ be a compact Riemann surface, and $L$ be a holomorphic line bundle over it. Let $\phi \in H^{0}(\Sigma,L)$  be not identically zero and $\Sigma$ be endowed with a metric whose associated $(1,1)$-form is $\chi$. We are interested in the following family of equations (for a Hermitian metric $h_{\alpha,t}$ on $L$) that depend on the parameters $\alpha\geq 0$ and  $0\leq t \leq 1$
	\begin{equation}
\label{1}	
i\Theta_{h_{\alpha,t}}=(d-\lvert\phi\rvert_{h_{\alpha,t}}^{2})\frac{e_{\alpha}u_{\alpha}^{1-t}\chi+itk_{\alpha}\nabla^{(1,0)}\phi\wedge\nabla^{(0,1)}\phi^{*}}{a_{\alpha}+b_{\alpha}t\lvert\phi\rvert_{h_{\alpha,t}}^{2}-c_{\alpha}t^{2}\lvert\phi\rvert_{h_{\alpha,t}}^{4}},
\end{equation}
	where $d>0,a_{\alpha}>0,e_{\alpha}>0,b_{\alpha},c_{\alpha},k_{\alpha}\geq 0$ are constants, $u_{\alpha}=\frac{a_{\alpha}i\Theta_{h_{\alpha,t=0}}}{e_{\alpha}(d-\lvert\phi\rvert_{h_{\alpha,t=0}}^{2})\chi}>0$ is a function, and $\Theta_{h_{\alpha,t}}$ is the curvature of the metric $h_{\alpha,t}.$ We choose a path in the $\alpha$-$t$ plane (lying in the region $\alpha\geq 0, 0< t\leq 1$) which starts at $(\alpha_{0},t_{0})$ and ends at $(\alpha_{1}, t_{1})$.% with $i\Theta_{\alpha_{0},t_{0}}=\omega_{\Sigma}$. 

\par
	Suppose $h_{\alpha,t}=h_{\alpha,t=0}e^{-\psi_{\alpha,t}}$ solves the above equation. The following theorem proves $C^{2,\beta}$ \textit{a priori} estimate on $\psi_{\alpha,t}$.  From now onwards, we suppress the dependence of $a_{\alpha},b_{\alpha},c_{\alpha},e_{\alpha},k_{\alpha}$ on $\alpha$.
	\begin{theorem} \label{2}
		Suppose $a,b,c,e,k$ depend on $\alpha$ continuously and $b-cd\geq 0$.Then the following statements hold.
\begin{enumerate}
\item If  $\lVert{\psi_{\alpha,t}}\rVert_{C^{0}} \leq C_{\alpha}$, then $\lVert{\psi_{\alpha,t}}\rVert_{C^{2,\beta}}<C,$ where $C_{\alpha}$ is independent of $t$.
\item If $b-(k+ct)d\geq0,$ then  ${\psi_{\alpha,t}} \geq -C'_{\alpha}$, where $C'_{\alpha}$ is independent of $t$.\\
Moreover, if we have $de>a$  and $i\Theta_{0}=\chi,$ then $\psi_{\alpha,t}\leq C''_{\alpha},$ where $C''_{\alpha}$ is independent of $t.$

\end{enumerate}
In addition, if the hypotheses above hold for all $\alpha\in [\alpha_0,\alpha_1]$ where $\alpha_0, \alpha_1>0$, then $C_{\alpha},$ $C'_{\alpha}$ and $C''_{\alpha}$ depend only on $\alpha_0, \alpha_1$.
	\end{theorem}
Theorem \ref{2} is applicable in a wide variety of situations as this paper hopes to demonstrate. \\ \par
The proof of theorem \ref{2} is similar to the one in \cite{vamsi2}. However the main aim of this result is the geometric consequences. The solution of the Calabi-Yang-Mills equations does not depend on the genus of the Riemann surface unlike in the case of K\"ahler-Yang-Mills equations in \cite{LMO2}(Theorem $1.1$). We also used our result to prove that Gieseker stability of a vector bundle is equivalent to the existence of an almost Hermitian Einstein metric, in a special case. We also prove K\"ahlerness of the negative of the symplectic form which arises in the moment map interpretation of the Calabi-Yang-Mills equations in \cite{Vamsi3} for small $\alpha$. The magnitude of the admissible $\alpha$ depends only on $\tau$ and the positivity of the Ricci curvature of $f$. We want to remark that there is a small computational gap in the uniqueness proof in \cite{vamsi2} which can be fixed following the calculations in this paper.\\
\indent In [\cite{LMO1},\cite{LMO2},\cite{LMOV}], the authors introduced the K\"ahler-Yang-Mills (KYM) equations to parametrise the moduli space of triples $(X,E,L)$ (of a polarised manifold with a holomorphic vector bundle over it) and proved existence, non-existence, and uniqueness results for the same. Solving the KYM equations in general, is quite challenging because they are of order four. Taking cue from the Calabi volume conjecture, Pingali \cite{Vamsi1} proposed to study the Calabi-Yang-Mills (CYM) equations as an easier toy model of the more complicated KYM equations and  provided a moment map interpretation for them\cite{Vamsi3}. In \cite{Vamsi1}, an openness result was proved for a special case of the CYM equations arising out of a vortex-type bundle. In more detail, let $X=\Sigma \times \mathbb{C}\mathbb{P}^{1}$. An action of $SU(2)$ on $X$ can be defined as follows : $SU(2)$ acts trivially on $\Sigma$ and in the standard manner on $\mathbb{C}\mathbb{P}^{1}=SU(2)/U(1)$. We now follow the calculations in [\cite{Vamsi1},\cite{Oscar1}]. Let $E$ be a rank-$2$ holomorphic vector bundle over $X$ defined as an extension :
	\[0\to \pi_{1}^{*}L\to E\to \pi_{2}^{*}\mathcal{O}(2)\to 0.\]
The second fundamental form of the extension is 
$\beta=\pi_{1}^{*}\phi\otimes\pi_{2}^{*}\xi,$
where $\xi=\frac{\sqrt{8\pi}dz}{(1+\lvert z\rvert^{2})^{2}}\otimes d\bar{z}.$\\ 
\indent Let $\tau>0$ be a constant and $\omega_{FS}=\frac{idz\wedge d\bar{z}}{(1+|z|^{2})^{2}}$ be the Fubini-Study metric on $\mathbb{C}\mathbb{P}^{1}.$ Denote by $\Omega=\pi_{1}^{*}\omega_{\Sigma}+\frac{4}{\tau}\omega_{FS}$, an $SU(2)$-invariant K\"ahler form on $X$ where $\int_{\Sigma}^{}\omega_{\Sigma}=vol(\Sigma)$ is fixed, by $H$, an $SU(2)$-invariant hermitian metric on $E$, and by $\Theta$ the curvature of $H$.
	 Then for this case the CYM equations amount to solving the following vortex-CYM equation for a Hermitian metric $h$ on $L$.
	 \begin{equation}
	 \label{25}
	 i\Theta_{h}=(\frac{\tau-\lvert\phi\rvert_{h}^{2}}{2})\frac{4f+\frac{i\alpha}{2(2\pi)^{2}}\nabla^{1,0}\phi\wedge\nabla^{0,1}\phi^{*}}{4+\frac{\tau\alpha}{(2\pi)^{2}}(2\lambda-\frac{\tau}{2})+\frac{\tau\alpha}{2(2\pi)^{2}}\lvert\phi\rvert_{h}^{2}-\frac{\alpha}{4(2\pi)^{2}}\lvert\phi\rvert_{h}^{4}},
	 \end{equation}
	 where $\alpha\geq0$ satisfies $8+\frac{2\tau\alpha}{(2\pi)^{2}}(2\lambda-\frac{\tau}{2})>0.$ In \ref{25} , $\Theta_{h}$ is the curvature of the metric $h$ on the line bundle $L$. In \cite{Vamsi1}, Pingali proved the set of $\alpha \geq 0$ satisfying
	 \[8+\frac{2\tau\alpha}{(2\pi)^{2}}(2\lambda-\frac{\tau}{2})>0\]
	 for which there exists a smooth form $\Omega_{\alpha}>0$
	 and a smooth metric $H_{\alpha}$ such
	 that the vortex-CYM equation is satisfied,
	 contains $\alpha=0$ and is open. Our first application of Theorem \ref{2} is the following result.
	 \begin{prop}
	 A	smooth solution of the vortex-CYM equation exists and is unique among all $SU(2)$-invariant solutions when $\alpha$ satisfies $8+\frac{2\tau\alpha}{(2\pi)^{2}}(2\lambda-\frac{\tau}{2})>0.$
	 \end{prop}  \par
 In \cite{Vamsi3} , Pingali gave the moment map interpretation of Calabi-Yang-Mills equations. For the vortex bundle ansatz, we prove the following. 
 \begin{theorem}
 	The negative of the symplectic form is K\"ahler whenever the Ricci curvature of $f$ is positive, $\tau<\frac{8}{3}$ and $\alpha$ is small.
\end{theorem} 
In \cite{LMOV}, the symplectic form is K\"ahler but here the symplectic form is not always K\"ahler. This phenomenon is happening probably because of the fact that the openness argument in \cite{Vamsi1} is not the standard integration-by-parts argument. The details and the proof are in sub-section \ref{kahlerness} . \par 
	 In a different development \cite{vamsi2}, Pingali introduced the vector bundle Monge-Amp\`ere (vbMA) equation  motivated by a desire to study stability conditions involving higher Chern forms. The vbMA equation for a metric $g'$ on a holomorphic vector bundle $F$ over  a compact complex $n$-dimensional manifold $Y$ is:
	 \begin{equation}
	 \label{26}
	 \left(\frac{i\Theta_{g'}}{2\pi}\right)^{n}=\eta Id,
	 \end{equation} 
	 where $\Theta_{g'}$ is the curvature of the Chern connection of $(F,g')$ and $\eta$ is a given volume form. The vbMA equation for vortex-type bundles akin to above, was studied in \cite{vamsi2}.  The \emph{a priori} estimates proved in \cite{vamsi2} follow as a direct corollary of Theorem \ref{2} (as indicated in Subsection \ref{subsec:vbMA}). \\
\begin{comment}
Suppose the $\omega_{\Sigma}=i\Theta_{0}$ where $\Theta_{0}$ is the curvature of a metric $h_{0}$ on a degree-$1$ holomorphic line bundle $L$. Let $\mathbb{C}\mathbb{P}^{1}$ be endowed with the Fubini-Study metric $\omega_{FS}$=$\frac{idz\wedge d\bar{z}}{(1+|z|^{2})^{2}}$ which is the curvature of a metric $h_{FS}$ on $\mathcal{O}(1)$. Consider the rank-$2$ bundle akin to the above. 
	 \[V=\pi_{1}^{*}((r_{1}+1)L)\otimes\pi_{2}^{*}(r_{2}\mathcal{O}(2))\oplus\pi_{1}^{*}(r_{1}L)\otimes\pi_{2}^{*}((r_{2}+1)\mathcal{O}(2)),\]
	 where $r_{1},r_{2}\geq 2.$ \par
	 Suppose $h$ is a smooth metric on $L$ and $f_{2}$ is a smooth positive function on $\Sigma$. Put a metric $H=h_{1}+g_{2}$ on $V$ where $h_{1}=\pi_{1}^{*}(hf_{2}h_{0}^{r_{1}})\otimes\pi_{2}^{*}(h_{FS}^{2r_{2}})$ is a metric on $\pi_{1}^{*}((r_{1}+1)L)\otimes\pi_{2}^{*}(r_{2}\mathcal{O}(2))$ and $g_{2}=\pi_{1}^{*}(f_{2}h_{0}^{r_{1}})\otimes\pi_{2}^{*}(h_{FS}^{2r_{2}+2})$ is a metric on $\pi_{1}^{*}(r_{1}L)\otimes\pi_{2}^{*}((r_{2}+1)\mathcal{O}(2))$. Using a holomorphic section $\phi\in H^{0}(\Sigma\times\mathbb{C}\mathbb{P}^{1},L)$ endow $V$ with a holomorphic structure through the second fundamental form $\beta$ as before. \par 
\end{comment}
\indent Finally, using some results in \cite{vamsi2} and Theorem \ref{2}, we  prove an existence and uniqueness result (Theorem \ref{giesekerresult}) for vortex-type bundles over a product of a Riemann surface and the sphere, relating Gieseker stability and almost Hermitian Einstein metrics. In more detail, let $E$ be a holomorphic vector bundle of rank $r$ over a compact K\"ahler manifold $X$ of dimension $n.$ Suppose $\omega$ is an integral form and therefore defines a line bundle $L'$ on $X$. The almost Hermitian Einstein equation is
	 \begin{equation}
	 \label{17}
	 [e^{(\frac{i}{2\pi}R_{E}+k\omega I_{E})}Td_{X}]^{TOP}=(const)I_{E}\frac{\omega^{n}}{n!},
	 \end{equation}
	 where $Td_{X}$ is the harmonic representative (with respect to $\omega$) of the Todd class. The constant $(const)$ is calculated  by taking the trace and integrating on both sides.
	 \begin{equation}
	 \label{18}
	 \chi(X,E\otimes L'^{k})\coloneqq \int Tr([e^{(\frac{i}{2\pi}R_{E}+k\omega I_{E})}Td_{X}]^{TOP}) =(const)r Vol(X).
	 \end{equation} 
	 So Equation \ref{17} becomes
	 \begin{equation}
	 \label{19}
	 [e^{(\frac{i}{2\pi}R_{E}+k\omega I_{E})}Td_{X}]^{TOP}=\frac{\chi(X,E\otimes L'^{k})}{r Vol(X)}I_{E}\frac{\omega^{n}}{n!}.
	 \end{equation} 
In \cite{Leung}, Leung claimed a general existence result for Equation \ref{17} for large $k$. Our result is not subsumed by Leung's claim because we provide an effective lower bound on $k$, and our result is equivariant in the sense that Gieseker stability only needs to be checked for $SU(2)$-invariant subbundles. Moreover, we prove uniqueness in the space of $SU(2)$- invariant solutions. The precise statement and proof of Theorem \ref{giesekerresult} is in Section \ref{gieskersec}. \\

\emph{Acknowledgements} : I thank my advisor, Vamsi Pritham Pingali, for suggesting this problem to me and for his constant encouragement. He helped me to correct several mistakes and make the paper more readable. The author is supported by a scholarship from the Indian Institute of Science. Lastly, the author is immensely thankful to the anonymous referee for their careful reading of the manuscript and useful suggestions.  
	\section{Proof of Theorem \ref{2}}
	In this section, we prove our main Theorem \ref{2}. Firstly, we have the following lemma which is useful in proving the desired estimates. 
	\begin{lemma}
		\label{3}
		If $b-cd\geq 0,$ then $\lvert \phi\rvert_{h_{t,\alpha}}^{2}\leq d$, for $0 \leq t \leq 1$.
	\end{lemma}
\begin{proof}
	We have the following identity\begin{equation}
	\label{4}	
	\partial\bar{\partial}\lvert\phi\rvert_{h_{t,\alpha}}^{2}=-\Theta_{h_{t,\alpha}}\lvert\phi\rvert_{h_{t,\alpha}}^{2}+\nabla_{t,\alpha}^{(1,0)}\phi\wedge\nabla_{t,\alpha}^{(0,1)}\phi^{*}.\end{equation}
	At the maximum point of $\lvert\phi\rvert_{h_{t,\alpha}}^{2}$(say $p$), we have $i\partial\bar{\partial}\lvert\phi\rvert_{h_{t,\alpha}}^{2}\leq 0$ and $\nabla_{t,\alpha}\lvert\phi\rvert_{h_{t,\alpha}}^{2}=0$. Therefore $\nabla_{t,\alpha}\phi(p)=0$ since $\phi$ is not identically zero. That gives us $i\Theta_{h_{t,\alpha}}(p)\geq 0$. Now we can write the denominator of equation \ref{1} as $a+t\lvert\phi\rvert_{h_{t,\alpha}}^{2}(b-ct\lvert\phi\rvert_{h_{t,\alpha}}^{2})$. Under the hypothesis, it is clear that $d-\lvert\phi\rvert_{h_{t,\alpha}}^{2}(p)\geq0$. Hence $\lvert\phi\rvert_{h_{t,\alpha}}^{2}(x)\leq\lvert\phi\rvert_{h_{t,\alpha}}^{2}(p)\leq d$.
\end{proof}
\vspace*{2mm}
From now onwards, we suppress the dependence of $\psi_{t,\alpha}$ on $t$ and $\alpha$.
\begin{lemma}
	If $\lVert\psi\rVert_{C^{1}}\leq C$, then $\lVert\psi\rVert_{C^{2,\beta}}\leq C$.
\end{lemma}
\begin{proof}
	Using the hypothesis and lemma \ref{3}, we see that the right-hand side of \ref{1} is uniformly bounded in $C^{0}$. Therefore, by $L^{p}$ regularity of elliptic equations, $\psi$ is bounded uniformly in $W^{2,p}$ for all large $p$. Using the Sobolev embedding theorem, we see that $\lVert\psi\rVert_{C^{1,\beta}}\leq C$. Thus the right-hand side is in $C^{0,\beta}$. Now using Schauder estimates, we are done.
\end{proof}
\vspace*{2mm}
The following lemma completes the proof of the first part of \ref{2}.
\begin{lemma}
	If $\lVert\psi\rVert_{C^{0}}\leq C,$ then $\lVert\psi\rVert_{C^{1}}\leq C$.
\end{lemma}
\begin{proof}
	To arrive at a contradiction, we assume that there exists a sequence $\psi_{n}$ (corresponding to $t_{n}$) such that $\lvert d\psi_{n}\rvert=\lvert d\psi_{n}(p_{n})\rvert=M_{n} \to \infty.$ We may assume $p_{n}\to p$(upto some subsequence). Now choose large enough $n$ so that $p_{n},p$ lie in a coordinate ball $B$ centered at $p$ with coordinate $z$(with $z=0$ corresponding to $p$). Define $\tilde{\psi_{n}}(\tilde{z})=\psi_{n}(p_{n}+\frac{\tilde{z}}{M_{n}})$. Now $\lvert \tilde{d}\tilde{\psi_{n}}\rvert \leq 1=\lvert\tilde{d} \tilde{\psi_{n}}\rvert(0)$. Note that
	\[\frac{\partial\tilde{\psi_{n}}}{\partial\tilde{z}}=\frac{1}{M_{n}}\frac{\partial\psi_{n}}{\partial z},\frac{\partial\tilde{\psi_{n}}}{\partial\bar{\tilde{z}}}=\frac{1}{M_{n}}\frac{\partial\psi_{n}}{\partial \bar{z}}\] 
 \begin{equation}
 \label{5}
 \frac{\partial^{2}\tilde{\psi_{n}}}{\partial\tilde{z}\partial\bar{\tilde{z}}}=\frac{1}{M_{n}^{2}}\frac{\partial^{2}\psi_{n}}{\partial z \partial\bar{z}}.\end{equation} 
 Now using \ref{5} and \ref{1}, we get
 \[i\Theta_{0}+i\frac{\partial^{2}\psi_{n}}{\partial z \partial\bar{z}}dz\wedge d\bar{z}=(d-\lvert\phi\rvert_{h_{n}}^{2})\frac{eu^{1-t_{n}}\chi+it_{n}k\nabla_{n}^{(1,0)}\phi\wedge\nabla_{n}^{(0,1)}\phi^{*}}{a+bt_{n}\lvert\phi\rvert_{h_{n}}^{2}-ct_{n}^{2}\lvert\phi\rvert_{h_{n}}^{4}}.\]
 We abuse notation from this point onwards and denote the functions $\frac{\chi}{idz\wedge d\bar{z}}=\frac{\Theta_{0}}{dz\wedge d\bar{z}}$ by $\chi$ and $\frac{\nabla_{n}^{(1,0)}\phi\wedge\nabla_{n}^{(0,1)}\phi^{*}}{dz\wedge d\bar{z}}$ by $\tilde{\nabla}_{n}^{(1,0)}\phi\wedge\tilde{\nabla}_{n}^{(0,1)}\phi^{*}$. \par
 So the above equation becomes
 \begin{equation}
 \label{6}
\frac{\chi}{M_{n}^{2}}+i\frac{\partial^{2}\tilde{\psi_{n}}}{\partial\tilde{z}\partial\bar{\tilde{z}}}=(d-\lvert\phi\rvert_{h_{n}}^{2})\frac{\frac{eu^{1-t_{n}}\chi}{M_{n}^{2}}+it_{n}k\tilde{\nabla}_{n}^{(1,0)}\phi\wedge\tilde{\nabla}_{n}^{(0,1)}\phi^{*}}{a+bt_{n}\lvert\phi\rvert_{h_{n}}^{2}-ct_{n}^{2}\lvert\phi\rvert_{h_{n}}^{4}}.
 \end{equation}
  Now the denominator in \ref{6} can be written as $a+t_{n}\lvert\phi\rvert_{h_{n}}^{2}(b-ct_{n}\lvert\phi\rvert_{h_{n}}^{2})$ and this shows that the denominator is bounded below by $a$ because $t_{n}\lvert\phi\rvert_{h_{n}}^{2}\leq \lvert\phi\rvert_{h_{n}}^{2}\leq d\leq\frac{b}{c}$. On a coordinate ball $B_{r}(0)$ in the $\tilde{z}$ coordinates, we have $\lvert\tilde{d}\tilde{\psi_{n}}\rvert \leq 1.$ Using \ref{6}, we conclude that $\lvert\tilde{\Delta}\tilde{\psi_{n}}\rvert\leq C$ on $B_{r}(0)$. Therefore, by interior $L^{p}$ regularity and the Sobolev embedding, we see that $\lVert\tilde{\psi_{n}}\rVert_{C^{1,\beta}(B_{0.7r}(0))}\leq C.$ Thus by the interior Schauder estimates $\lVert\tilde{\psi_{n}}\rVert_{C^{2,\beta}(B_{0.5r}(0))}\leq C.$ Suppose $\lVert\tilde{\psi_{n}}\rVert_{C^{2,\beta}(B_{0.5r}(0))}\leq C_{r}$ for some fixed $\beta>0$. For every fixed $r$, a subsequence of $\tilde{\psi_{n}}$ converges in $C^{2,\gamma}(B_{0.5r}(0))$ to a function $\tilde{\psi_{r}}$ for a fixed $\gamma<\beta.$   Choosing a diagonal subsequence, we may assume that for all $r$, we have a single function $\tilde{\psi}.$ Now it is easy to see using \ref{6} that $i\frac{\partial^{2}\tilde{\psi}}{\partial\tilde{z}\partial\bar{\tilde{z}}}\geq 0$. But a subharmonic function on $\mathbb{C}$ cannot be bounded above unless it is a constant. Hence $\tilde{\psi}$ is a constant. But this contradicts the fact that $\lvert\tilde{d}\tilde{\psi}\rvert(0)=1$. Hence $\lvert\nabla \psi\rvert \leq C$, thus implying a $C^{1}$ estimate.
\end{proof} 
\vspace*{3mm}

To prove the second part of theorem \ref{2}, that is the $C^{0}$ estimate, we need the following form of the Green representation formula. Let $G$ be a Green function of the background metric $\chi$ such that $-C[1+\lvert ln(d_{\chi}(P,Q))\rvert]\leq G(P,Q)\leq 0$. Then any function $f$ satisfies the following equation
\begin{equation}
\label{7}
f(Q)=\frac{\int f\chi}{\int \chi}+\int_{\Sigma}^{}G(P,Q)i\partial\bar{\partial}f(P).
\end{equation}
Now we prove a lower bound on $\psi$.
\begin{lemma}
	\label{8}
	If $b-(k+ct)d\geq 0,$ then the function $\psi$ satisfies $\psi\geq -C,$ where $C$ is independent of $t$. 
\end{lemma}
\begin{proof}
	Firstly note that $\lvert\phi\rvert_{h}^{2}\leq d$. That is $\lvert\phi\rvert_{h_{0}}^{2}\leq e^{\psi+d'}$(where $e^{d'}=d$). Therefore 
	\begin{equation}
	\label{9}
	\int\psi\chi\geq \int ln(\lvert\phi\rvert_{h_{0}}^{2})\chi-\int d'\chi.
	\end{equation}
	Now 
	\[
	\psi(P)\geq \frac{\int ln(\lvert\phi\rvert_{h_{0}}^{2})\chi}{\int \chi}-d'+\int_{\Sigma}G(Q,P)i\partial\bar{\partial}\psi(Q)
	\]
	\begin{equation}
	\label{10}
	\psi(P)\geq \frac{\int ln(\lvert\phi\rvert_{h_{0}}^{2})\chi}{\int \chi}-d'+\int_{\Sigma}G(Q,P)i\Theta_{h}(Q)-\int G(Q,P)\chi.
	\end{equation}
	From \ref{1} we get
	\[i\Theta_{h}\leq d \frac{eu^{1-t}\chi+itk\nabla^{(1,0)}\phi\wedge\nabla^{(0,1)}\phi^{*}}{a+tb\lvert\phi\rvert_{h}^{2}-ct^{2}\lvert\phi\rvert_{h}^{2}(\lvert\phi\rvert_{h}^{2}-d)-dct^{2}\lvert\phi\rvert_{h}^{2}}\]
	\vspace{1mm}
	\[\implies i\Theta_{h}\leq d \frac{eu^{1-t}\chi+itk\nabla^{(1,0)}\phi\wedge\nabla^{(0,1)}\phi^{*}}{a+t\lvert\phi\rvert_{h}^{2}(b-ctd)}.\] 
	
	Now $b-ctd\geq 0$ because $b-cd\geq 0$. Therefore using \ref{4} , we have
	\[i\Theta_{h}(a+t\lvert\phi\rvert_{h}^{2}(b-ctd))\leq deu^{1-t}\chi+iktd\partial\bar{\partial}\lvert\phi\rvert_{h}^{2}+iktd\lvert\phi\rvert_{h}^{2}\Theta_{h}\]	
	\[\implies i\Theta_{h}(a+t\lvert\phi\rvert_{h}^{2}(b-ctd-kd))\leq deu^{1-t}\chi+ iktd\partial\bar{\partial}\lvert\phi\rvert_{h}^{2}.\]
	We need the hypothesis for the following inequality to hold.
	\[ i\Theta_{h} a\leq  i\Theta_{h}(a+t\lvert\phi\rvert_{h}^{2}(b-ctd-kd))\leq deu^{1-t}\chi+ iktd\partial\bar{\partial}\lvert\phi\rvert_{h}^{2}\]
	\begin{equation}
	\label{11}	
		\implies G(Q,P)i\Theta_{h}\geq G(Q,P)\frac{deu^{1-t}\chi+ iktd\partial\bar{\partial}\lvert\phi\rvert_{h}^{2}}{a}.
	\end{equation}	
	Using \ref{10} and \ref{11} we get
	\[\psi(P)\geq \frac{\int ln(\lvert\phi\rvert_{h_{0}}^{2})\chi}{\int \chi}-d'-\int G(Q,P)\chi+\int G(Q,P)\frac{deu^{1-t}\chi+ iktd\partial\bar{\partial}\lvert\phi\rvert_{h}^{2}}{a}(Q).\]
	Now using Green representation formula \ref{7} we get
	\[\implies \psi(P)\geq \frac{\int ln(\lvert\phi\rvert_{h_{0}}^{2})\chi}{\int \chi}-d'-\int G(Q,P)\chi+\int G(Q,P)\frac{deu^{1-t}\chi}{a}+\frac{ktd}{a}(\lvert\phi\rvert_{h}^{2}(P)-\frac{\int\lvert\phi\rvert_{h}^{2}\chi}{\int\chi})\geq -C.\]
	Hence, we are done.
	
\end{proof}
Now we prove an upper bound on $\psi$. 
\begin{lemma}
	\label{12}
	If $de>a$ and $i\Theta_{0}=\chi,$ then $\psi\leq C,$ where $C$ is independent of $t$.
\end{lemma}
\begin{proof}
	Suppose $\psi$ achieves its maximum value at a point $p$. Then at that point, we have $i\partial\bar{\partial}\psi\leq 0,$ and $\partial \psi=0=\bar{\partial}\psi.$ Now from equation \ref{1} we have
	\begin{equation}
	\label{13}
	i\Theta_{0}=\chi\geq (d-\lvert\phi\rvert_{h}^{2})\frac{eu^{1-t}\chi}{a+bt\lvert\phi\rvert_{h}^{2}-ct^{2}\lvert\phi\rvert_{h}^{4}}.
	\end{equation}
	If the upper bound does not hold and suppose there exists a sequence $\psi_{n}(p_{n})\to \infty$(with $p_{n}\to q$) then $\lvert\phi\rvert_{h_{n}}^{2}(q)\to 0$. Hence from \ref{13} we get	
	\[1\geq \frac{de}{a}u^{1-t}\geq \frac{de}{a}\frac{a^{1-t}}{(ed)^{1-t}}=\frac{(de)^{t}}{a^{t}}>1.\]
	We have a contradiction since $0<t\leq 1$. Hence $\psi\leq C$. 
\end{proof}
This completes the proof of the theorem \ref{2}. 
\section{Three Applications.}
In this section, we apply our result in three cases.  
\subsection{Calabi Yang Mills Equations.}\label{subsec:CYM}
\subsubsection{\emph{A Priori} Estimates For Calabi Yang Mills Equations.}
In \cite{Vamsi1}, Pingali considered the Calabi-Yang-Mills equations(theorem 1.2). If we rewrite the equation (3.29) of \cite{Vamsi1}, then it becomes the following one:
\begin{equation}
\label{22}
i\Theta_{h}=(\frac{\tau-\lvert\phi\rvert_{h}^{2}}{2})\frac{4f+\frac{i\alpha}{2(2\pi)^{2}}\nabla^{1,0}\phi\wedge\nabla^{0,1}\phi^{*}}{4+\frac{\tau\alpha}{(2\pi)^{2}}(2\lambda-\frac{\tau}{2})+\frac{\tau\alpha}{2(2\pi)^{2}}\lvert\phi\rvert_{h}^{2}-\frac{\alpha}{4(2\pi)^{2}}\lvert\phi\rvert_{h}^{4}}.
\end{equation}
We consider the following continuity path with parameter $\alpha$ :
\begin{equation}
\label{28}
i\Theta_{h_{\alpha}}=(\frac{\tau-\lvert\phi\rvert_{h_{\alpha}}^{2}}{2})\frac{4f+\frac{i\alpha}{2(2\pi)^{2}}\nabla_{h_{\alpha}}^{1,0}\phi\wedge\nabla_{h_{\alpha}}^{0,1}\phi^{*}}{4+\frac{\tau\alpha}{(2\pi)^{2}}(2\lambda-\frac{\tau}{2})+\frac{\tau\alpha}{2(2\pi)^{2}}\lvert\phi\rvert_{h_{\alpha}}^{2}-\frac{\alpha}{4(2\pi)^{2}}\lvert\phi\rvert_{h_{\alpha}}^{4}}.
\end{equation} 
Here the only difference is the curvature of the initial metric. The curvature of the initial metric is $i\Theta_{0}=(\frac{\tau-\lvert\phi\rvert_{h_{0}}^{2}}{2})\omega_{\Sigma}$(which can be seen from equation $3.8$ of \cite{Vamsi1}). In \cite{Vamsi1}, Pingali proved the set of $\alpha \geq 0$ satisfying
\[8+\frac{2\tau\alpha}{(2\pi)^{2}}(2\lambda-\frac{\tau}{2})>0,\]
for which there exists a smooth form $\Omega_{\alpha}>0$
and a smooth metric $H_{\alpha}$ such
that the vortex-CYM equation is satisfied,
contains $\alpha=0$ and is open. So we can assume that our path starts at $\alpha_{0}>0.$ If we compare the above equations \ref{28} with the equations we considered \ref{1}(with $t=1$), then $a=8+\frac{2\tau\alpha}{(2\pi)^{2}}(2\lambda-\frac{\tau}{2}), b=\frac{\tau\alpha}{(2\pi)^{2}}, c=\frac{\alpha}{2(2\pi)^{2}}, d=\tau,e=4, k=\frac{\alpha}{2(2\pi)^{2}}.$ We see that $a>0,d>0,e>0,b,c,k\geq 0.$ Now $b-cd=\frac{\tau\alpha}{2(2\pi)^{2}}\geq 0$ and $ b-(k+c)d=0.$ This shows that all the estimates hold except for the upper bound one because here the curvature of the initial metric is not of the form, we considered. However, the following lemma proves the upper bound estimate.
\begin{lemma}
	\label{15}
	If $h_{\alpha}=h_{0}e^{-\psi_{\alpha}}$ solves \ref{28} for $\alpha_{0}<\alpha \leq \alpha_{1}$ where $\alpha_{1}$ satisfies $8+\frac{2\tau\alpha_{1}}{(2\pi)^{2}}(2\lambda-\frac{\tau}{2})>0$ , then $\psi_{\alpha}\leq C,$ where $C$ is independent of $\alpha$. Here, $h_{0}$ denotes the metric corresponding to $\alpha=0$.
\end{lemma}
\begin{proof}
	Suppose $\psi$ achieves its maximum value at a point $p.$ Then at that point, we have $i\partial\bar{\partial}\psi\leq 0,$ and $\partial \psi=0=\bar{\partial}\psi.$ Now from equation \ref{28}, we have
	\begin{equation}
	\label{16}
	i\Theta_{0}\geq (\tau-\lvert\phi\rvert_{h_{\alpha}}^{2})\frac{4f}{8+\frac{2\tau\alpha}{(2\pi)^{2}}(2\lambda-\frac{\tau}{2})+ \frac{\tau\alpha}{(2\pi)^{2}}\lvert\phi\rvert_{h_{\alpha}}^{2}-\frac{\alpha}{2(2\pi)^{2}}\lvert\phi\rvert_{h_{\alpha}}^{4}}.
	\end{equation}
	If the upper bound does not hold and suppose there exists a sequence $\psi_{n}(p_{n})\to \infty$(with $p_{n}\to q$) then $\lvert\phi\rvert_{h_{n}}^{2}(q)\to 0$. Hence from \ref{16}, we get
		\[i\Theta_{0}(q)\geq \tau \frac{4f}{8+\frac{2\tau\alpha}{(2\pi)^{2}}(2\lambda-\frac{\tau}{2})}\]
		\[\implies \frac{\tau-\lvert\phi\rvert_{h_{0}}^{2}}{2}f\geq\tau \frac{4f}{8+\frac{2\tau\alpha}{(2\pi)^{2}}(2\lambda-\frac{\tau}{2})} \]
		\[\implies \frac{\tau f}{2} \geq \tau \frac{4f}{8+\frac{2\tau\alpha}{(2\pi)^{2}}(2\lambda-\frac{\tau}{2})}\]
		\[\implies 1\geq  \frac{2e}{a} \]
 Now $\frac{2e}{a}>1$ because $2\lambda-\frac{\tau}{2}<0$, which follows from \cite{Vamsi1}. So we have a contradiction. Hence $\psi\leq C$.
\end{proof}
So we now have closedness and hence existence for all those $\alpha$ satisfying $8+\frac{2\tau\alpha}{(2\pi)^{2}}(2\lambda-\frac{\tau}{2})>0.$
\subsubsection{Uniqueness Of Solutions Of Calabi-Yang-Mills Equations.}
We prove that for all those $\alpha$ for which \ref{22} has a solution is essentially unique among all $SU(2)$-invariant solutions. Our proof of uniqueness is as follows.\\
 Let $h_{1}$ be the solution arising from \ref{22},  that is
\[i\Theta_{h_{1}}=(\frac{\tau-\lvert\phi\rvert_{h_{1}}^{2}}{2})\frac{4f+\frac{i\alpha}{2(2\pi)^{2}}\nabla_{h_{1}}^{1,0}\phi\wedge\nabla_{h_{1}}^{0,1}\phi^{*}}{4+\frac{\tau\alpha}{(2\pi)^{2}}(2\lambda-\frac{\tau}{2})+\frac{\tau\alpha}{2(2\pi)^{2}}\lvert\phi\rvert_{h_{1}}^{2}-\frac{\alpha}{4(2\pi)^{2}}\lvert\phi\rvert_{h_{1}}^{4}}.\]
Let $h_{2}$ denote any other solution of equation \ref{22}. We run the continuity path backward with continuity parameter $\beta$ starting with $h_{2}:$
\begin{equation}
\label{21}
i\Theta_{\tilde{h}}=(\frac{\tau-\lvert\phi\rvert_{\tilde{h}}^{2}}{2})\frac{4f+\frac{i\beta}{2(2\pi)^{2}}\nabla_{\tilde{h}}^{1,0}\phi\wedge\nabla_{\tilde{h}}^{0,1}\phi^{*}}{4+\frac{\tau\beta}{(2\pi)^{2}}(2\lambda-\frac{\tau}{2})+\frac{\tau\beta}{2(2\pi)^{2}}\lvert\phi\rvert_{\tilde{h}}^{2}-\frac{\beta}{4(2\pi)^{2}}\lvert\phi\rvert_{\tilde{h}}^{4}}.
\end{equation}
Denote $\tilde{T}\subset [0,\alpha]$ the set of $\beta$ such that $\ref{21}$ has a solution. This is non-empty because $\alpha\in \tilde{T}$. The proof of openness in \cite{Vamsi1}, shows that $\tilde{T}\subset [0,\alpha]$ is open. % The \textit{a priori} estimates for \ref{22} show that $\tilde{T}\subset [0,\alpha]$ is closed. Hence $\tilde{T}=[0,\alpha]$. So we are done if we can show that the solution at $\beta=0$ for \ref{22}
%is  unique. But in \cite{Vamsi1}, it is proved that for $\beta=0$ the solution is essentially unique among all $SU(2)$-invariant solutions. 
We prove that there exists a ``small" $\beta_{0} \in [0,\alpha]$ such that \ref{22} has a unique solution for $\beta_{0}$. That is, there exists a unique smooth $h$ satisfying $\lvert\phi \rvert_{h}^{2} \leq \tau$ and 

\[
i\Theta_{h}=(\frac{\tau-\lvert\phi\rvert_{h}^{2}}{2})\frac{4f+\frac{i\beta_{0}}{2(2\pi)^{2}}\nabla^{1,0}\phi\wedge\nabla^{0,1}\phi^{*}}{4+\frac{\tau\beta_{0}}{(2\pi)^{2}}(2\lambda-\frac{\tau}{2})+\frac{\tau\beta_{0}}{2(2\pi)^{2}}\lvert\phi\rvert_{h}^{2}-\frac{\beta_{0}}{4(2\pi)^{2}}\lvert\phi\rvert_{h}^{4}}.
\] 

This implies that the two continuity path \ref{21} and \ref{28} intersect at $\beta_{0}$. 

\begin{lemma}
	There exists a number $\beta_{0} \in (0,\alpha]$ depending only on $\lambda,\tau,\alpha$ such that there is a unique smooth $h$ satisfying $\lvert\phi\rvert_{h}^{2}\leq \tau$ and the following equation 
	\[i\Theta_{h}=(\frac{\tau-\lvert\phi\rvert_{h}^{2}}{2})\frac{4f+\frac{i\beta_{0}}{2(2\pi)^{2}}\nabla_{h}^{1,0}\phi\wedge\nabla_{h}^{0,1}\phi^{*}}{4+\frac{\tau\beta_{0}}{(2\pi)^{2}}(2\lambda-\frac{\tau}{2})+\frac{\tau\beta_{0}}{2(2\pi)^{2}}\lvert\phi\rvert_{h}^{2}-\frac{\beta_{0}}{4(2\pi)^{2}}\lvert\phi\rvert_{h}^{4}}.\]
\end{lemma}

\begin{proof}
	Let $h_{1}$ be the solution coming from the forward path \ref{28} and $h_{2}$ be the solution coming from the backward path \ref{21} , satisfying $\lvert\phi\rvert_{h_{2}}^{2}\leq \tau$. We define a function $g$ to satisfy $h_{2}=h_{1}e^{-g}$. Let $h_{s}=h_{1}e^{-sg}=h_{2}^{s}h_{1}^{1-s}$, where $0\leq s \leq 1$. It is easy to see that
	\[\lvert\phi\rvert_{h_{s}}^{2}=(\lvert\phi\rvert_{h_{1}}^{2(1-s)})(\lvert\phi\rvert_{h_{2}}^{2s})\leq \tau.\]
	
	Let $I_{s}=m+ l\beta_{0}\lvert\phi\rvert_{h_{s}}^{2}-q\beta_{0}\lvert\phi\rvert_{h_{s}}^{4}$ and $J_{s}=\frac{uf+iv\beta_{0}\nabla_{h_{s}}^{1,0}\phi\wedge\nabla_{h_{s}}^{0,1}\phi^{*}}{I_{s}}$, wehre $m=8+\frac{2\tau\beta_{0}}{(2\pi)^{2}}(2\lambda-\frac{\tau}{2}), l=\frac{\tau\beta_{0}}{(2\pi)^{2}}, q=\frac{\beta_{0}}{2(2\pi)^{2}}, u=4, v=\frac{1}{2(2\pi)^{2}}.$ \\
	So $i\Theta_{h_{s}}=(\tau-\lvert\phi\rvert_{h_{s}}^{2})J_{s}$
	
	By assumption $i\Theta_{h_{1}}=(\tau-\lvert\phi\rvert_{h_{1}})J_{1}$ and $i\Theta_{h_{2}}=(\tau-\lvert\phi\rvert_{h_{2}})J_{2}$. \\
	Now 
	
	\[i\partial\bar{\partial}g=i\Theta_{h_{2}}-i\Theta_{h_{1}}\]
	\[=\int_{0}^{1}ds \frac{d}{ds}(i\Theta_{h_{s}})\]
	\[=\int_{0}^{1}ds \frac{d}{ds}((\tau-\lvert \phi\rvert_{h_{s}}^{2})J_{s})\]
	\begin{equation}
	\label{29}
	=\int_{0}^{1}ds\{g\lvert\phi\rvert_{h_{s}}^{2}J_{s}+(\tau-\lvert\phi\rvert_{h_{s}}^{2})\frac{dJ_{s}}{ds}\}
	\end{equation}
	We now calculate $\frac{dJ_{s}}{ds}$.
	
	\[\frac{dJ_{s}}{ds}=\frac{iv\beta_{0}\frac{d}{ds}(\nabla_{h_{s}}^{1,0}\phi\wedge\nabla_{h_{s}}^{0,1}\phi^{*})}{I_{s}}-\frac{uf+iv\beta_{0}\nabla_{h_{s}}^{1,0}\phi\wedge\nabla_{h_{s}}^{0,1}\phi^{*}}{I_{s}^{2}}(-lg\lvert\phi\rvert_{h_{s}}^{2}+2qg\lvert\phi\rvert_{h_{s}}^{4})\beta_{0}\]
	
	\[=\frac{iv\beta_{0}\frac{d}{ds}(\partial\bar{\partial}\lvert\phi\rvert_{h_{s}}^{2}+\Theta_{h_{s}}\lvert\phi\rvert_{h_{s}}^{2})}{I_{s}}-\frac{uf+iv\beta_{0}\nabla_{h_{s}}^{1,0}\phi\wedge\nabla_{h_{s}}^{0,1}\phi^{*}}{I_{s}^{2}}(-lg\lvert\phi\rvert_{h_{s}}^{2}+2qg\lvert\phi\rvert_{h_{s}}^{4})\beta_{0}\]
	
	\[=\frac{iv\beta_{0}(\partial\bar{\partial}(-g\lvert\phi\rvert_{h_{s}}^{2})-g\Theta_{h_{s}}\lvert\phi\rvert_{h_{s}}^{2}+\lvert\phi\rvert_{h_{s}}^{2}\partial\bar{\partial}g)}{I_{s}}-\frac{uf+iv\beta_{0}\nabla_{h_{s}}^{1,0}\phi\wedge\nabla_{h_{s}}^{0,1}\phi^{*}}{I_{s}^{2}}(-lg\lvert\phi\rvert_{h_{s}}^{2}+2qg\lvert\phi\rvert_{h_{s}}^{4})\beta_{0}\]
	
	\[=\frac{iv\beta_{0}(-\partial\lvert\phi\rvert_{h_{s}}^{2}\wedge\bar{\partial}g-\partial g\wedge \bar{\partial}\lvert\phi\rvert_{h_{s}}^{2}-g\partial\bar{\partial}\lvert\phi\rvert_{h_{s}}^{2}-g\Theta_{h_{s}}\lvert\phi\rvert_{h_{s}}^{2})}{I_{s}}-\frac{uf+iv\beta_{0}\nabla_{h_{s}}^{1,0}\phi\wedge\nabla_{h_{s}}^{0,1}\phi^{*}}{I_{s}^{2}}(-lg\lvert\phi\rvert_{h_{s}}^{2}+2qg\lvert\phi\rvert_{h_{s}}^{4})\beta_{0}.\]
	
	Putting $\frac{dJ_{s}}{ds}$ in \ref{29}, we have
	\[i\partial\bar{\partial}g=\int_{0}^{1}ds\{g\lvert\phi\rvert_{h_{s}}^{2}J_{s}+(\tau-\lvert\phi\rvert_{h_{s}}^{2})\frac{iv\beta_{0}(-\partial\lvert\phi\rvert_{h_{s}}^{2}\wedge\bar{\partial}g-\partial g\wedge \bar{\partial}\lvert\phi\rvert_{h_{s}}^{2}-g\partial\bar{\partial}\lvert\phi\rvert_{h_{s}}^{2}-g\Theta_{h_{s}}\lvert\phi\rvert_{h_{s}}^{2})}{I_{s}}\}\]
	
	\[-\int_{0}^{1}ds\{(\tau-\lvert\phi\rvert_{h_{s}}^{2})\frac{uf+iv\beta_{0}\nabla_{h_{s}}^{1,0}\phi\wedge\nabla_{h_{s}}^{0,1}\phi^{*}}{I_{s}^{2}}(-lg\lvert\phi\rvert_{h_{s}}^{2}+2qg\lvert\phi\rvert_{h_{s}}^{4})\beta_{0}\}\] 
	
\begin{equation}
\label{30}
\begin{split}
&=\int_{0}^{1}ds\{g\lvert\phi\rvert_{h_{s}}^{2}J_{s}+(\tau-\lvert\phi\rvert_{h_{s}}^{2})\frac{iv\beta_{0}(-\partial\lvert\phi\rvert_{h_{s}}^{2}\wedge\bar{\partial}g-\partial g\wedge \bar{\partial}\lvert\phi\rvert_{h_{s}}^{2}-g\nabla_{h_{s}}^{1,0}\phi\wedge\nabla_{h_{s}}^{0,1}\phi^{*})}{I_{s}}\}\\
&-\int_{0}^{1}ds\{(\tau-\lvert\phi\rvert_{h_{s}}^{2})\frac{uf+iv\beta_{0}\nabla_{h_{s}}^{1,0}\phi\wedge\nabla_{h_{s}}^{0,1}\phi^{*}}{I_{s}^{2}}(-lg\lvert\phi\rvert_{h_{s}}^{2}+2qg\lvert\phi\rvert_{h_{s}}^{4})\beta_{0}\}
\end{split}
\end{equation}
We know that $g\geq -C_{1}$, where $C_{1}$ depends on $\lambda,\tau,\alpha$.\\
Now define $\tilde{g}=g(\gamma +\lvert\phi\rvert_{h}^{2})$, where $\gamma>1$ is a large constant (depending only on $\lambda,\tau,\alpha$) to be chosen later on and $h$ is defined as
\[h=\int_{0}^{1}h_{s}ds=h_{1}\int_{0}^{1}e^{-sg}ds.\]
It follows that $\lvert\phi\rvert_{h}^{2}\leq \tau.$ 
If maximum of $\tilde{g}$ occurs at $p$, then 
\[\partial\tilde{g}(p)=0=\bar{\partial}\tilde{g}(p).\]
Therefore 
\[\partial g(p)(\gamma+\lvert\phi\rvert_{h}^{2})(p)=-g(p)\partial(\lvert\phi\rvert_{h}^{2})(p)\]
\begin{equation}
\label{31}
\implies \partial g(p)=-\frac{g(p)\partial(\lvert\phi\rvert_{h}^{2})(p)}{\gamma+\lvert\phi\rvert_{h}^{2}(p)}.
\end{equation}

And similarly
\begin{equation}
\label{32}
\bar{\partial} g(p)=-\frac{g(p)\bar{\partial}(\lvert\phi\rvert_{h}^{2})(p)}{\gamma+\lvert\phi\rvert_{h}^{2}(p)}.
\end{equation}
Moreover, $i\partial\bar{\partial}\tilde{g}(p)\leq 0$, i.e., 
\[0\geq (\gamma+\lvert\phi\rvert_{h}^{2})i\partial\bar{\partial} g(p)+i \partial g(p)\wedge\bar{\partial}\lvert\phi\rvert_{h}^{2}(p)+i\partial \lvert\phi\rvert_{h}^{2}(p)\wedge\bar{\partial}g(p)+g(p)(-i\Theta_{h}(p)\lvert\phi\rvert_{h}^{2}(p)+i\nabla_{h}^{1,0}\phi(p)\wedge\nabla_{h}^{0,1}\phi^{*}(p)).\]

Using \ref{31}, \ref{32} and $\partial\lvert\phi\rvert_{h}^{2}\wedge\bar{\partial}\lvert\phi\rvert_{h}^{2}=\lvert\phi\rvert_{h}^{2}\nabla_{h}^{1,0}\phi\wedge\nabla_{h}^{0,1}\phi^{*}$ we have 
\[0\geq (\gamma+\lvert\phi\rvert_{h}^{2}(p))i\partial\bar{\partial} g(p)-\frac{2ig(p)}{\gamma+\lvert\phi\rvert_{h}^{2}(p)}\lvert\phi\rvert_{h}^{2}(p)\nabla_{h}^{1,0}\phi(p)\wedge\nabla_{h}^{0,1}\phi^{*}(p)+ig(p)\nabla_{h}^{1,0}\phi(p)\wedge\nabla_{h}^{0,1}\phi^{*}(p)-ig(p)\Theta_{h}(p)\lvert\phi\rvert_{h}^{2}(p)\]
\begin{equation}
\begin{split}
&\implies 0\geq i\partial\bar{\partial} g(p)-\frac{2ig(p)}{(\gamma+\lvert\phi\rvert_{h}^{2}(p))^{2}}\lvert\phi\rvert_{h}^{2}(p)\nabla_{h}^{1,0}\phi(p)\wedge\nabla_{h}^{0,1}\phi^{*}(p)\\
&+\frac{ig(p)}{\gamma+\lvert\phi\rvert_{h}^{2}(p)}\nabla_{h}^{1,0}\phi(p)\wedge\nabla_{h}^{0,1}\phi^{*}(p)-\frac{ig(p)}{\gamma+\lvert\phi\rvert_{h}^{2}(p)}\Theta_{h}(p)\lvert\phi\rvert_{h}^{2}(p)
\end{split}
\end{equation}

\begin{equation}
\begin{split}
&\implies 0\geq i\partial\bar{\partial} g(p)+\frac{ig(p)(\gamma-\lvert\phi\rvert_{h}^{2}(p))}{(\gamma+\lvert\phi\rvert_{h}^{2}(p))^{2}}\nabla_{h}^{1,0}\phi(p)\wedge\nabla_{h}^{0,1}\phi^{*}(p)\\
&-\frac{g(p)(\tau-\lvert\phi\rvert_{h}^{2}(p))uf(p)}{(\gamma+\lvert\phi\rvert_{h}^{2}(p))(m+l\beta_{0}\lvert\phi\rvert_{h}^{2}(p)-q\beta_{0}\lvert\phi\rvert_{h}^{4}(p))}\lvert\phi\rvert_{h}^{2}(p)\\
&-\frac{g(p)(\tau-\lvert\phi\rvert_{h}^{2}(p))iv\beta_{0}}{(m+l\beta_{0}\lvert\phi\rvert_{h}^{2}(p)-q\beta_{0}\lvert\phi\rvert_{h}^{4}(p))(\gamma+\lvert\phi\rvert_{h}^{2}(p))}\lvert\phi\rvert_{h}^{2}(p)\nabla_{h}^{1,0}\phi(p)\wedge\nabla_{h}^{0,1}\phi^{*}(p).
\end{split}
\end{equation}

Now using \ref{30} and suppressing the dependence on $p$, we have 

\begin{equation}
\begin{split}
&\implies 0\geq \int_{0}^{1}ds\{g\lvert\phi\rvert_{h_{s}}^{2}J_{s}+(\tau-\lvert\phi\rvert_{h_{s}}^{2})\frac{iv\beta_{0}(-\partial\lvert\phi\rvert_{h_{s}}^{2}\wedge\bar{\partial}g-\partial g\wedge \bar{\partial}\lvert\phi\rvert_{h_{s}}^{2}-g\nabla_{h_{s}}^{1,0}\phi\wedge\nabla_{h_{s}}^{0,1}\phi^{*})}{I_{s}}\}\\
&-\int_{0}^{1}ds\{(\tau-\lvert\phi\rvert_{h_{s}}^{2})\frac{uf+iv\beta_{0}\nabla_{h_{s}}^{1,0}\phi\wedge\nabla_{h_{s}}^{0,1}\phi^{*}}{I_{s}^{2}}(-lg\lvert\phi\rvert_{h_{s}}^{2}+2qg\lvert\phi\rvert_{h_{s}}^{4})\beta_{0}\}\\
&+\{\frac{ig(\gamma-\lvert\phi\rvert_{h}^{2})}{(\gamma+\lvert\phi\rvert_{h}^{2})^{2}}\ -\frac{g(\tau-\lvert\phi\rvert_{h}^{2})iv\beta_{0}\lvert\phi\rvert_{h}^{2}}{(m+l\beta_{0}\lvert\phi\rvert_{h}^{2}-q\beta_{0}\lvert\phi\rvert_{h}^{4})(\gamma+\lvert\phi\rvert_{h}^{2})}\}\nabla_{h}^{1,0}\phi\wedge\nabla_{h}^{0,1}\phi^{*}\\
&-\frac{g(\tau-\lvert\phi\rvert_{h}^{2})uf}{(\gamma+\lvert\phi\rvert_{h}^{2})(m+l\beta_{0}\lvert\phi\rvert_{h}^{2}-q\beta_{0}\lvert\phi\rvert_{h}^{4})}\lvert\phi\rvert_{h}^{2}.
\end{split}
\end{equation}
Using \ref{31}, \ref{32} and $J_{s}=\frac{uf+iv\beta_{0}\nabla_{h_{s}}^{1,0}\phi\wedge\nabla_{h_{s}}^{0,1}\phi^{*}}{I_{s}}$, we get the following inequality 
\begin{equation}
\begin{split}
& \implies 0 \geq \int_{0}^{1}ds\{\frac{g\lvert\phi\rvert_{h_{s}}^{2}uf}{I_{s}}-\frac{(\tau -\lvert\phi\rvert_{h_{s}}^{2})uf}{I_{s}^{2}}(-lg\lvert\phi\rvert_{h_{s}}^{2}+2qg\lvert\phi\rvert_{h_{s}}^{4})\beta_{0}\}\\
&+\int_{0}^{1}ds\{\frac{g\lvert\phi\rvert_{h_{s}}^{2}iv\beta_{0}}{I_{s}}-(\tau-\lvert\phi\rvert_{h_{s}}^{2})\frac{iv\beta_{0}g}{I_{s}}-\frac{iv\beta_{0}^{2}(\tau-\lvert\phi\rvert_{h_{s}}^{2})(-lg\lvert\phi\rvert_{h_{s}}^{2}+2qg\lvert\phi\rvert_{h_{s}}^{4})}{I_{s}^{2}}\}\nabla_{h_{s}}^{1,0}\phi\wedge\nabla_{h_{s}}^{0,1}\phi^{*}\\
&+\int_{0}^{1}ds\{\frac{iv\beta_{0}g(\tau -\lvert\phi\rvert_{h_{s}}^{2})}{I_{s}(\gamma+\lvert\phi\rvert_{h}^{2})}(\partial\lvert\phi\rvert_{h_{s}}^{2}\wedge\bar{\partial}\lvert\phi\rvert_{h}^{2}+\partial\lvert\phi\rvert_{h}^{2}\wedge\bar{\partial}\lvert\phi\rvert_{h_{s}}^{2})\}\\
&+\{\frac{ig(\gamma-\lvert\phi\rvert_{h}^{2})}{(\gamma+\lvert\phi\rvert_{h}^{2})^{2}}\ -\frac{g(\tau-\lvert\phi\rvert_{h}^{2})iv\beta_{0}\lvert\phi\rvert_{h}^{2}}{(m+l\beta_{0}\lvert\phi\rvert_{h}^{2}-q\beta_{0}\lvert\phi\rvert_{h}^{4})(\gamma+\lvert\phi\rvert_{h}^{2})}\}\nabla_{h}^{1,0}\phi\wedge\nabla_{h}^{0,1}\phi^{*}\\
&-\frac{g(\tau-\lvert\phi\rvert_{h}^{2})uf}{(\gamma+\lvert\phi\rvert_{h}^{2})(m+l\beta_{0}\lvert\phi\rvert_{h}^{2}-q\beta_{0}\lvert\phi\rvert_{h}^{4})}\lvert\phi\rvert_{h}^{2}
\end{split}
\end{equation}

\begin{equation}
\label{33}
\begin{split}
& \implies 0 \geq \int_{0}^{1}ds\{\frac{g\lvert\phi\rvert_{h_{s}}^{2}uf}{I_{s}}+\frac{(\tau -\lvert\phi\rvert_{h_{s}}^{2})uf\beta_{0}g\lvert\phi\rvert_{h_{s}}^{2}}{I_{s}^{2}}(l-2q\lvert\phi\rvert_{h_{s}}^{2})\}
-\frac{g(\tau-\lvert\phi\rvert_{h}^{2})uf}{(\gamma+\lvert\phi\rvert_{h}^{2})(m+l\beta_{0}\lvert\phi\rvert_{h}^{2}-q\beta_{0}\lvert\phi\rvert_{h}^{4})}\lvert\phi\rvert_{h}^{2}\\
&+\int_{0}^{1}ds\{\frac{iv\beta_{0}g}{I_{s}}(2\lvert\phi\rvert_{h_{s}}^{2}-\tau)+\frac{iv\beta_{0}^{2}g\lvert\phi\rvert_{h_{s}}^{2}}{I_{s}^{2}}(\tau-\lvert\phi\rvert_{h_{s}}^{2})(l-2q\lvert\phi\rvert_{h_{s}}^{2})\}\nabla_{h_{s}}^{1,0}\phi\wedge\nabla_{h_{s}}^{0,1}\phi^{*}\\
&+\int_{0}^{1}ds\{\frac{iv\beta_{0}g(\tau -\lvert\phi\rvert_{h_{s}}^{2})}{I_{s}(\gamma+\lvert\phi\rvert_{h}^{2})}(\partial\lvert\phi\rvert_{h_{s}}^{2}\wedge\bar{\partial}\lvert\phi\rvert_{h}^{2}+\partial\lvert\phi\rvert_{h}^{2}\wedge\bar{\partial}\lvert\phi\rvert_{h_{s}}^{2})\}\\
&+\{\frac{ig(\gamma-\lvert\phi\rvert_{h}^{2})}{(\gamma+\lvert\phi\rvert_{h}^{2})^{2}}\ -\frac{g(\tau-\lvert\phi\rvert_{h}^{2})iv\beta_{0}\lvert\phi\rvert_{h}^{2}}{(m+l\beta_{0}\lvert\phi\rvert_{h}^{2}-q\beta_{0}\lvert\phi\rvert_{h}^{4})(\gamma+\lvert\phi\rvert_{h}^{2})}\}\nabla_{h}^{1,0}\phi\wedge\nabla_{h}^{0,1}\phi^{*}.
\end{split}
\end{equation}
The following equations descibe the relationship between $\nabla_{h_{s}}\phi$ and $\nabla_{h}\phi$.
\begin{equation}
\label{34}
\begin{split}
&\nabla_{h_{s}}^{1,0}\phi\wedge\nabla_{h_{s}}^{0,1}\phi^{*}=(\nabla_{h}^{1,0}\phi-\partial\ln(\int_{0}^{1}e^{-tg}dt)-s\partial g \phi)\wedge\nabla^{0,1}(\phi^{*}e^{-sg}(\int_{0}^{1}e^{-tg}dt)^{-1})\\
&=\frac{e^{-sg}}{\int_{0}^{1}e^{-tg}dt}(\nabla_{h}^{1,0}\phi+(\langle s\rangle-s)\partial g\phi)\wedge(\nabla^{0,1}\phi^{*}+(\langle s\rangle-s)\bar{\partial}g\phi^{*})\\
&=\frac{e^{-sg}}{\int_{0}^{1}e^{-tg}dt}\nabla_{h}^{1,0}\phi \wedge \nabla_{h}^{0,1}\phi^{*}\left(1+\frac{(s-\langle s\rangle)g\lvert\phi\rvert_{h}^{2}}{\gamma+\lvert\phi\rvert_{h}^{2}}\right)^{2},
\end{split}
\end{equation}
where $\langle s \rangle =\frac{\int_{0}^{1}se^{-sg}ds}{\int_{0}^{1}e^{-sg}ds} \leq 1.$\\

Note that
\begin{equation}
\label{41}
\left\lvert\frac{(s-\langle s\rangle)g\lvert\phi\rvert_{h}^{2}}{\gamma+\lvert\phi\rvert_{h}^{2}}\right\rvert \leq  \frac{2}{\gamma}\lvert g\rvert\lvert\phi\rvert_{h_{1}}^{2}\int_{0}^{1}e^{-tg}dt\\
\leq \frac{2\tau}{\gamma}\lvert 1- e^{-g}\rvert\\
\leq\frac{2\tau}{\gamma}(1+e^{C_{1}}).
\end{equation}

We argue by contradiction. Assume that 
\begin{equation}
\label{40}
g(p)\geq 0.
\end{equation}
 We see that $l-2q\lvert\phi\rvert_{h_{s}}^{2}=\frac{\beta_{0}}{(2\pi)^{2}}(\tau-\lvert\phi\rvert_{h_{s}}^{2})\geq 0$.\\ Using \ref{34} and \ref{40}, the inequality \ref{33} becomes the following
\begin{equation}
\begin{split}
& \implies 0 \geq \int_{0}^{1}ds\{\frac{g\lvert\phi\rvert_{h_{s}}^{2}uf}{m+l\beta_{0}\tau}+\frac{(\tau -\lvert\phi\rvert_{h_{s}}^{2})uf\beta_{0}g\lvert\phi\rvert_{h_{s}}^{2}}{(m+l\beta_{0}\tau)^{2}}(l-2q\lvert\phi\rvert_{h_{s}}^{2})\}
-\frac{g\tau uf\lvert\phi\rvert_{h}^{2}}{(\gamma+\lvert\phi\rvert_{h}^{2})(m-q\beta_{0}\tau^{2})}\\
&+\int_{0}^{1}ds\{\frac{iv\beta_{0}g}{I_{s}}(2\lvert\phi\rvert_{h_{s}}^{2}-\tau)+\frac{iv\beta_{0}^{2}g\lvert\phi\rvert_{h_{s}}^{2}}{I_{s}^{2}}(\tau-\lvert\phi\rvert_{h_{s}}^{2})(l-2q\lvert\phi\rvert_{h_{s}}^{2})\}\frac{e^{-sg}}{\int_{0}^{1}e^{-tg}dt}\nabla_{h}^{1,0}\phi \wedge \nabla_{h}^{0,1}\phi^{*}\left(1+\frac{(s-\langle s\rangle)g\lvert\phi\rvert_{h}^{2}}{\gamma+\lvert\phi\rvert_{h}^{2}}\right)^{2}\\
&+\int_{0}^{1}ds\{\frac{iv\beta_{0}g(\tau -\lvert\phi\rvert_{h_{s}}^{2})}{I_{s}(\gamma+\lvert\phi\rvert_{h}^{2})}(\partial\lvert\phi\rvert_{h_{s}}^{2}\wedge\bar{\partial}\lvert\phi\rvert_{h}^{2}+\partial\lvert\phi\rvert_{h}^{2}\wedge\bar{\partial}\lvert\phi\rvert_{h_{s}}^{2})\}\\
&+\{\frac{ig(\gamma-\lvert\phi\rvert_{h}^{2})}{(\gamma+\lvert\phi\rvert_{h}^{2})^{2}}\ -\frac{g(\tau-\lvert\phi\rvert_{h}^{2})iv\beta_{0}\lvert\phi\rvert_{h}^{2}}{(m+l\beta_{0}\lvert\phi\rvert_{h}^{2}-q\beta_{0}\lvert\phi\rvert_{h}^{4})(\gamma+\lvert\phi\rvert_{h}^{2})}\}\nabla_{h}^{1,0}\phi\wedge\nabla_{h}^{0,1}\phi^{*}
\end{split}
\end{equation}

\begin{equation}
\begin{split}
\label{35}
& \implies 0 \geq \frac{g\lvert\phi\rvert_{h}^{2}uf}{m+l\beta_{0}\tau}-\frac{g\tau uf\lvert\phi\rvert_{h}^{2}}{(\gamma+\lvert\phi\rvert_{h}^{2})(m-q\beta_{0}\tau^{2})}+\int_{0}^{1}\frac{(\tau -\lvert\phi\rvert_{h_{s}}^{2})uf\beta_{0}g\lvert\phi\rvert_{h_{s}}^{2}}{(m+l\beta_{0}\tau)^{2}}(l-2q\lvert\phi\rvert_{h_{s}}^{2})ds\\
&+\int_{0}^{1}ds\{\frac{iv\beta_{0}g}{I_{s}}(2\lvert\phi\rvert_{h_{s}}^{2}-\tau)+\frac{iv\beta_{0}^{2}g\lvert\phi\rvert_{h_{s}}^{2}}{I_{s}^{2}}(\tau-\lvert\phi\rvert_{h_{s}}^{2})(l-2q\lvert\phi\rvert_{h_{s}}^{2})\}\frac{e^{-sg}}{\int_{0}^{1}e^{-tg}dt}\nabla_{h}^{1,0}\phi \wedge \nabla_{h}^{0,1}\phi^{*}\left(1+\frac{(s-\langle s\rangle)g\lvert\phi\rvert_{h}^{2}}{\gamma+\lvert\phi\rvert_{h}^{2}}\right)^{2}\\
&+\int_{0}^{1}ds\{\frac{iv\beta_{0}g(\tau -\lvert\phi\rvert_{h_{s}}^{2})}{I_{s}(\gamma+\lvert\phi\rvert_{h}^{2})}(\partial\lvert\phi\rvert_{h_{s}}^{2}\wedge\bar{\partial}\lvert\phi\rvert_{h}^{2}+\partial\lvert\phi\rvert_{h}^{2}\wedge\bar{\partial}\lvert\phi\rvert_{h_{s}}^{2})\}\\
&+\{\frac{ig(\gamma-\lvert\phi\rvert_{h}^{2})}{(\gamma+\lvert\phi\rvert_{h}^{2})^{2}}-\frac{g(\tau-\lvert\phi\rvert_{h}^{2})iv\beta_{0}\lvert\phi\rvert_{h}^{2}}{(m+l\beta_{0}\lvert\phi\rvert_{h}^{2}-q\beta_{0}\lvert\phi\rvert_{h}^{4})(\gamma+\lvert\phi\rvert_{h}^{2})}\}\nabla_{h}^{1,0}\phi\wedge\nabla_{h}^{0,1}\phi^{*}.
\end{split}
\end{equation}

Now
\[\int_{0}^{1}\left\lvert\partial\lvert\phi\rvert_{h}^{2}\wedge\bar{\partial}\lvert\phi\rvert_{h_{s}}^{2}\right\rvert ds\]
\[\leq \int_{0}^{1}\left\lvert\partial\lvert\phi\rvert_{h}^{2}\right\rvert \left\lvert\bar{\partial}\lvert\phi\rvert_{h_{s}}^{2}\right\rvert \omega_{\Sigma}  ds\]
\[\leq \int_{0}^{1}\lvert\phi\rvert_{h}\lvert\phi\rvert_{h_{s}}\left\lvert\nabla_{h}^{1,0}\phi\right\rvert \left\lvert\nabla_{h_{s}}^{0,1}\phi^{*}\right\rvert \omega_{\Sigma} ds\] 

\[\leq \tau\int_{0}^{1}\left\lvert\nabla_{h}^{1,0}\phi\right\rvert \left\lvert\nabla_{h}^{0,1}\phi^{*} -\bar{\partial}\ln (\int_{0}^{1}e^{-gt}dt)\phi^{*}-s\bar{\partial} g \phi^{*} \right\rvert \omega_{\Sigma} ds\]

\[\leq \tau\int_{0}^{1}\left\lvert\nabla_{h}^{1,0}\phi\right\rvert \left\lvert\nabla_{h}^{0,1}\phi^{*} +(\langle s \rangle -s)\bar{\partial} g \phi^{*} \right\rvert \frac{\sqrt{e^{-sg}}}{\sqrt{\int_{0}^{1}e^{-tg}dt}}\omega_{\Sigma} ds.\]
Using \ref{32}, we have 

\[\leq \tau \left\lvert\nabla_{h}^{1,0}\phi\right\rvert \left\lvert\nabla_{h}^{0,1}\phi^{*}\right\rvert \int_{0}^{1}\left\lvert\left(1+\frac{(s-\langle s\rangle)g\lvert\phi\rvert_{h}^{2}}{\gamma+\lvert\phi\rvert_{h}^{2}}\right)\right\rvert\frac{\sqrt{e^{-sg}}}{\sqrt{\int_{0}^{1}e^{-tg}dt}}\omega_{\Sigma} ds.\]
Putting \ref{41}, we get
\[\leq \tau \left\lvert\nabla_{h}^{1,0}\phi\right\rvert \left\lvert\nabla_{h}^{0,1}\phi^{*}\right\rvert \int_{0}^{1}\left(1+\frac{2\tau}{\gamma}(1+e^{C_{1}})\right)\frac{\sqrt{e^{-sg}}}{\sqrt{\int_{0}^{1}e^{-tg}dt}}\omega_{\Sigma} ds\]

\[\leq \tau \left\lvert\nabla_{h}^{1,0}\phi\right\rvert \left\lvert\nabla_{h}^{0,1}\phi^{*}\right\rvert\left(1+\frac{2\tau}{\gamma}(1+e^{C_{1}})\right) \int_{0}^{1}\frac{\sqrt{e^{-sg}}}{\sqrt{\int_{0}^{1}e^{-tg}dt}}\omega_{\Sigma} ds.\]
Using Cauchy-Schwarz inequality, we have
\begin{equation}
\label{36}
\int_{0}^{1}\left\lvert\partial\lvert\phi\rvert_{h}^{2}\wedge\bar{\partial}\lvert\phi\rvert_{h_{s}}^{2}\right\rvert ds \leq \tau \left\lvert\nabla_{h}^{1,0}\phi\right\rvert \left\lvert\nabla_{h}^{0,1}\phi^{*}\right\rvert \left(1+\frac{2\tau}{\gamma}(1+e^{C_{1}})\right)\omega_{\Sigma}
\end{equation}
and similarly
\begin{equation}
\label{37}
\int_{0}^{1}\left\lvert\partial\lvert\phi\rvert_{h_{s}}^{2}\wedge\bar{\partial}\lvert\phi\rvert_{h}^{2}\right\rvert ds \leq \tau \left\lvert\nabla_{h}^{1,0}\phi\right\rvert \left\lvert\nabla_{h}^{0,1}\phi^{*}\right\rvert \left(1+\frac{2\tau}{\gamma}(1+e^{C_{1}})\right) \omega_{\Sigma}.
\end{equation}
We can write \ref{35} as
\begin{equation}
\label{38}
0 \geq A guf+ Big \nabla_{h}^{1,0}\phi\wedge\nabla_{h}^{0,1}\phi^{*}+\int_{0}^{1}ds\{\frac{iv\beta_{0}g(\tau -\lvert\phi\rvert_{h_{s}}^{2})}{I_{s}(\gamma+\lvert\phi\rvert_{h}^{2})}(\partial\lvert\phi\rvert_{h_{s}}^{2}\wedge\bar{\partial}\lvert\phi\rvert_{h}^{2}+\partial\lvert\phi\rvert_{h}^{2}\wedge\bar{\partial}\lvert\phi\rvert_{h_{s}}^{2})\}
\end{equation}
where 
\[A= \frac{\lvert\phi\rvert_{h}^{2}}{m+l\beta_{0}\tau}-\frac{\tau \lvert\phi\rvert_{h}^{2}}{(\gamma+\lvert\phi\rvert_{h}^{2})(m-q\beta_{0}\tau^{2})}+\int_{0}^{1}\frac{(\tau -\lvert\phi\rvert_{h_{s}}^{2})\beta_{0}\lvert\phi\rvert_{h_{s}}^{2}}{(m+l\beta_{0}\tau)^{2}}(l-2q\lvert\phi\rvert_{h_{s}}^{2})ds\]
and
\begin{equation}
	\label{B}
\begin{split}
&B=\int_{0}^{1}ds\left[\frac{v\beta_{0}}{I_{s}}(2\lvert\phi\rvert_{h_{s}}^{2}-\tau)+\frac{v\beta_{0}^{2}\lvert\phi\rvert_{h_{s}}^{2}}{I_{s}^{2}}(\tau-\lvert\phi\rvert_{h_{s}}^{2})(l-2q\lvert\phi\rvert_{h_{s}}^{2})\right]\frac{e^{-sg}}{\int_{0}^{1}e^{-tg}dt}\left(1+\frac{(s-\langle s\rangle)g\lvert\phi\rvert_{h}^{2}}{\gamma+\lvert\phi\rvert_{h}^{2}}\right)^{2}\\
&+\left[\frac{(\gamma-\lvert\phi\rvert_{h}^{2})}{(\gamma+\lvert\phi\rvert_{h}^{2})^{2}}\ -\frac{(\tau-\lvert\phi\rvert_{h}^{2})v\beta_{0}\lvert\phi\rvert_{h}^{2}}{(m+l\beta_{0}\lvert\phi\rvert_{h}^{2}-q\beta_{0}\lvert\phi\rvert_{h}^{4})(\gamma+\lvert\phi\rvert_{h}^{2})}\right].
\end{split}
\end{equation}
We recall that $l-2q\lvert\phi\rvert_{h_{s}}^{2}=\frac{\beta_{0}}{(2\pi)^{2}}(\tau-\lvert\phi\rvert_{h_{s}}^{2})\geq 0$. Now we can choose $\gamma$ (depending only on $\tau,\lambda,\alpha$) large enough so that
 \[A\geq 0.\] 
 \ref{41} implies that \ref{B} is bounded and \ref{36} , \ref{37} implies that \\
 	$i\int_{0}^{1}ds\{\frac{v\beta_{0}(\tau -\lvert\phi\rvert_{h_{s}}^{2})}{I_{s}(\gamma+\lvert\phi\rvert_{h}^{2})}(\partial\lvert\phi\rvert_{h_{s}}^{2}\wedge\bar{\partial}\lvert\phi\rvert_{h}^{2}+\partial\lvert\phi\rvert_{h}^{2}\wedge\bar{\partial}\lvert\phi\rvert_{h_{s}}^{2})\}$ is bounded. Now
 	\begin{equation}
 		iB\nabla_{h}^{1,0}\phi\wedge\nabla_{h}^{0,1}\phi^{*}+i\int_{0}^{1}ds\{\frac{v\beta_{0}(\tau -\lvert\phi\rvert_{h_{s}}^{2})}{I_{s}(\gamma+\lvert\phi\rvert_{h}^{2})}(\partial\lvert\phi\rvert_{h_{s}}^{2}\wedge\bar{\partial}\lvert\phi\rvert_{h}^{2}+\partial\lvert\phi\rvert_{h}^{2}\wedge\bar{\partial}\lvert\phi\rvert_{h_{s}}^{2})\}
 	\end{equation}
 	  is positive when $\beta_{0}=0$. So for small $\beta_{0}$, we have the following.
\begin{equation}
\label{39}
iB\nabla_{h}^{1,0}\phi\wedge\nabla_{h}^{0,1}\phi^{*}+i\int_{0}^{1}ds\{\frac{v\beta_{0}(\tau -\lvert\phi\rvert_{h_{s}}^{2})}{I_{s}(\gamma+\lvert\phi\rvert_{h}^{2})}(\partial\lvert\phi\rvert_{h_{s}}^{2}\wedge\bar{\partial}\lvert\phi\rvert_{h}^{2}+\partial\lvert\phi\rvert_{h}^{2}\wedge\bar{\partial}\lvert\phi\rvert_{h_{s}}^{2})\} \geq 0.
\end{equation}
Since the line bundle is of degree $1$, either $\phi(p) \neq 0$ or $\lvert\nabla\phi\rvert(p)\neq 0$. This implies that $g(p)\leq 0$ which contradicts \ref{40}. Therefore $g\leq 0$. The same argument applied to a point of minimum of $g$ shows that $g\geq 0$. Hence $g=0$ showing uniqueness for small $\beta_{0}$. 
\end{proof}
The inequality $4.65$ in \cite{vamsi2} is not true. As remarked in the  introduction, the gap can be fixed using the method done in this paper.\\
We now complete the proof of uniqueness.
\begin{lemma}
	If there exists a $\beta_{0}\in [0,1]$ such that $h_{\beta_{0}}=\tilde{h}_{\beta_{0}}$ then $h_{1}=h_{2}$.
\end{lemma}
\begin{proof}
	Let $T\subset[\beta_{0},1]$ be the set of all $\alpha$ such that $h_{\alpha}=\tilde{h}_{\alpha}$. Then $T$ satisfies the following.\\
		(1) It is non-empty : $\beta_{0}\in T$\\
		(2) It is open : The proof of openness(see \cite{Vamsi1}) and the Inverse Function Theorem of Banach manifolds shows that locally the solution is unique and hence $T$ is open.\\
		(3) It is closed : The \emph{a priori} estimates show that $T$ is closed.
		
		Therefore $T=[\beta_{0},1]$.	
\end{proof}
\subsection{$J$-Vortex Equation}
In \cite{Takahashi}, Takahashi introduced $J$-equation on holomorphic vector bundles and came up with the $J-$ vortex equation($6.21$ in \cite{Takahashi}). The continuity path ($6.28$ in \cite{Takahashi}) in \cite{Takahashi} is 
  \[F_{h_{t}}=\omega_{\Sigma}+\frac{\sqrt{-1}}{2\pi}\partial\bar{\partial}\psi_{t}=2(1-\lvert\phi\rvert_{h_{t}}^{2})\frac{s u^{1-t}\omega_{\Sigma}+\frac{\sqrt{-1}c'^{2}t}{\pi}D_{t}'\phi D_{t}''\phi^{*t}}{(4c'r_{2}-2c't\lvert\phi\rvert_{h_{t}}^{2}-1+4c')(4c'r_{2}+2c't\lvert\phi\rvert_{h_{t}}^{2}-1)},\]
  where $u=\frac{1}{\alpha(1-\lvert\phi\rvert_{h_{\Sigma}}^{2})}$ and $\alpha=\frac{2s}{(4c'r_{2}-1+4c')(4c'r_{2}-1)}$. We replaced $c$ by $c'$ in the continuity path to avoid confusion in notation. This is exactly the type of equations that we considered. Here $a=(4c'r_{2}-1)(4c'r_{2}-1+4c'), b=8c'^{2}, c=4c'^{2}, e=2s>0,k=\frac{2c'^{2}}{\pi}, d=1$, where $c'=\frac{2r_{1}+1+2s(2r_{2}+1)}{4(2r_{1}+1)r_{2}+4r_{1}(r_{2}+1)}$, $r_{1},r_{2}$ are positive integers and $s$ is a positive real number. Now $b-cd=4c'^{2}>0$ and $b-(k+ct)d=4c'^{2}(1-t)+4c'^{2}-\frac{2c'^{2}}{\pi}\geq 0$ and $de>a$ which follows from lemma $6.17$ in \cite{Takahashi}. So the \textit{a priori} estimates for this equations follow. 
\subsection{Vector Bundle Version Of The Monge Amp\`{e}re Equation.}\label{subsec:vbMA}
In \cite{vamsi2}, Pingali considered the vector bundle version of the Monge Amp\`{e}re equation. The continuity path $(4.20)$ in \cite{vamsi2} is 
\[i\Theta_{h_{t}}=(1-\lvert\phi\rvert_{h_{t}}^{2})\frac{\mu u^{1-t}\omega_{\Sigma}+it\nabla_{t}^{(1,0)}\phi\wedge\nabla_{t}^{(0,1)}\phi^{*}}{(2r_{2}+t\lvert\phi\rvert_{h_{t}}^{2})(2+2r_{2}-t\lvert\phi\rvert_{h_{t}}^{2})},\]
where $u=\frac{1}{\alpha(1-\lvert\phi\rvert_{h_{0}}^{2})}$ and $\alpha=\frac{\mu}{2r_{2}(2+2r_{2})}, \mu=2(2r_{1}r_{2}+r_{1}+r_{2})$. This is exactly the type of equations that we considered. Here $a=2r_{2}(2+2r_{2}),b=2,c=1,d=1,e=\mu, k=1$ and $b-cd>0, b-(k+ct)d\geq 0,$ and $de>a$ since $r_{1}>r_{2}$. So the \textit{a priori} estimates for this equations follow.  

\section{Gieseker Stability and Almost Hermitian Einstein Metric} \label{gieskersec}

Consider a genus-$g$ compact Riemann surface $\Sigma$ endowed with a metric whose $(1,1)$ form $\omega_{\Sigma}=i\Theta_{0}$, where $\Theta_{0}$ is the curvature of a metric $h_{0}$ on a degree $1$ line bundle $L$. Let $\mathbb{C}\mathbb{P}^{1}$ be endowed with the Fubini-Study metric $\omega_{FS}$ which is the curvature of a metric $h_{FS}$ on $\mathcal{O}(1)$.

\par 
Consider the rank-$2$ vector bundle 
\[E=\pi_{1}^{*}((r_{1}+1)L)\otimes\pi_{2}^{*}(r_{2}\mathcal{O}(2))\oplus\pi_{1}^{*}(r_{1}L)\otimes\pi_{2}^{*}((r_{2}+1)\mathcal{O}(2)),\]
where $r_{1},r_{2}\geq 2$ and $\pi_{1},\pi_{2}$ are projections from $\Sigma\times\mathbb{C}\mathbb{P}^{1}$ to $\Sigma$ and $\mathbb{C}\mathbb{P}^{1}$ respectively. Endow $E$ with a holomorphic structure arising from the second fundamental form $\beta$ just like in Section \ref{sec:intro}. \par

Now we calculate $[e^{(\frac{i}{2\pi}R_{E}+k\omega I_{E})}Td_{X}]^{TOP}$ for the bundle $E$ over the manifold $\Sigma\times\mathbb{C}\mathbb{P}^{1}$ with $\omega=\frac{\tau}{2}\omega_{\Sigma}+2\omega_{FS},$ where $\tau$ is an even integer.
We know that \[Td(\Sigma\times\mathbb{C}\mathbb{P}^{1})=1+\frac{c_{1}(\Sigma)}{2}+\frac{c_{1}(\mathbb{C}\mathbb{P}^{1})}{2}+\frac{c_{1}(\Sigma)\wedge c_{1}(\mathbb{C}\mathbb{P}^{1})}{4},\]
\[e^{(\frac{i}{2\pi}R_{E}+k\omega I_{E})}=(1+k\omega)I_{E}+\frac{i}{2\pi}(1+k\omega)\wedge R_{E}+\frac{1}{2(2\pi)^{2}}(iR_{E})^{2}+\tau k^{2}\omega_{\Sigma}\wedge\omega_{FS}I_{E}.\]
Therefore
\[
[e^{(\frac{i}{2\pi}R_{E}+k\omega I_{E})}Td_{X}]^{TOP}=\frac{c_{1}(\Sigma)\wedge c_{1}(\mathbb{C}\mathbb{P}^{1})}{4}I_{E}+\frac{kc_{1}(\Sigma)\wedge \omega}{2}I_{E}+\frac{kc_{1}(\mathbb{C}\mathbb{P}^{1})\wedge\omega}{2}I_{E}+\frac{i}{2\pi}k\omega\wedge R_{E}+\frac{i}{2(2\pi)}c_{1}(\Sigma)\wedge R_{E}\]
\[ 
+\frac{i}{2(2\pi)}c_{1}(\mathbb{C}\mathbb{P}^{1})\wedge R_{E} + k^{2}\tau \omega_{\Sigma}\wedge\omega_{FS}I_{E}+\frac{(iR_{E})^{2}}{2(2\pi)^{2}}.\]
Here we note some equalities. $c_{1}(\Sigma)=\alpha\omega_{\Sigma},$ $c_{1}(\mathbb{C}\mathbb{P}^{1})=2\omega_{FS},$ $c_{1}(S)=(r_{1}+1)\omega_{\Sigma}+2r_{2}\omega_{FS},$ $c_{1}(E)=(2r_{1}+1)\omega_{\Sigma}+(4r_{2}+2)\omega_{FS},ch_{2}(E)=2\{(r_{1}+1)r_{2}+r_{1}(r_{2}+1)\}\omega_{\Sigma}\wedge\omega_{FS}, ch_{2}(S)=2r_{2}(r_{1}+1)\omega_{\Sigma}\wedge\omega_{FS}$ where $\alpha=2-2g$,$S=\pi_{1}^{*}((r_{1}+1)L)\otimes\pi_{2}^{*}(r_{2}\mathcal{O}(2))$. \\
We now recall the definition of Gieseker stability.
\begin{defn}
	Let $E$ be a rank $r$ holomorphic vector bundle over a projective variety $X$ with ample line bundle $L',$ $E$ is called Gieseker stable if for any nontrivial coherent subsheaf $S$ of $E,$ we have 
	\[\frac{\chi(X,S\otimes L'^{k})}{rankS}<\frac{\chi(X,E\otimes L'^{k})}{rankE}\]
	for large enough $k.$
\end{defn}
\begin{theorem}\label{giesekerresult}
	If $k$ satisfies $k(\tau-2)+(\alpha-1)+2(r_{1}-r_{2})>0$ and $k\tau+\alpha>0$, then the following are equivalent.\\
	(1) $E$ is Gieseker stable.\\
	(2) There exists an almost Hermitian Einstein metric on $E$.\\
	Moreover, the solution is unique among all $SU(2)-$invariant solutions.
\end{theorem}
\begin{proof}
	$2\Rightarrow 1$ follows from \cite{Leung}. We only prove $1\Rightarrow 2$. To this end, we only need to use the Gieseker stability assumption for the $SU(2)$-invariant subbundle $S=\pi_{1}^{*}((r_{1}+1)L)\otimes\pi_{2}^{*}(r_{2}\mathcal{O}(2))$. The assumption reads as follows.
	\begin{equation}
		\label{20}
		2\chi(X,S\otimes L'^{k})<\chi(X,E\otimes L'^{k}).
	\end{equation}
	Now 
	\[\chi(X,E\otimes L'^{k})=\int Tr([e^{(\frac{i}{2\pi}R_{E}+k\omega I_{E})}Td_{X}]^{TOP})\]
	\[= \int Tr(\frac{c_{1}(\Sigma)\wedge c_{1}(\mathbb{C}\mathbb{P}^{1})}{4}I_{E}+\frac{kc_{1}(\Sigma)\wedge \omega}{2}I_{E}+\frac{kc_{1}(\mathbb{C}\mathbb{P}^{1})\wedge\omega}{2}I_{E}+\frac{i}{2\pi}k\omega\wedge R_{E}+\frac{i}{2(2\pi)}c_{1}(\Sigma)\wedge R_{E})\]
	\[ 
	+\int Tr(\frac{i}{2(2\pi)}c_{1}(\mathbb{C}\mathbb{P}^{1})\wedge R_{E} + k^{2}\tau \omega_{\Sigma}\wedge\omega_{FS}I_{E}+\frac{(iR_{E})^{2}}{2(2\pi)^{2}}).\]
	%	\[=\int\alpha\omega_{\Sigma}\wedge\omega_{FS}+\int 2\alpha k\omega_{\Sigma}\wedge\omega_{FS}+\int k\tau\omega_{\Sigma}\wedge\omega_{FS}+\int k(\tau(2r_{2}+2)+2(2r_{1}+1)\omega_{\Sigma}\wedge\omega_{FS}\]
	%	\[+\int \alpha(2r_{2}+2)\omega_{\Sigma}\wedge\omega_{FS}+\int (2r_{1}+1)\omega_{\Sigma}\wedge\omega_{FS}+\int 2k^{2}\tau \omega_{\Sigma}\wedge\omega_{FS}+\int ch_{2}(E).\]
	Hence Inequality \ref{20} becomes 
	\[2\int k\omega\wedge c_{1}(S)+\int c_{1}(\Sigma)\wedge c_{1}(S)+\int c_{1}(\mathbb{C}\mathbb{P}^{1})\wedge c_{1}(S)+2\int ch_{2}(S)\]
	\[<\int k\omega\wedge c_{1}(E)+\frac{1}{2}\int c_{1}(\Sigma)\wedge c_{1}(E)+\frac{1}{2}\int c_{1}(\mathbb{C}\mathbb{P}^{1})\wedge c_{1}(E)+\int ch_{2}(E)\]
	\[\implies 2k\int (r_{2}\tau+2(r_{1}+1))\omega_{\Sigma}\wedge\omega_{FS}+\int2\alpha r_{2}\omega_{\Sigma}\wedge\omega_{FS}+\int 2(r_{1}+1)\omega_{\Sigma}\wedge\omega_{FS}+2\int ch_{2}(S)\]
	\[<k\int \{(2r_{2}+1)\tau+2(2r_{1}+1)\}\omega_{\Sigma}\wedge\omega_{FS}+\frac{1}{2}\int \alpha(4r_{2}+2)\omega_{\Sigma}\wedge\omega_{FS}+\frac{1}{2}\int 2(2r_{1}+1)\omega_{\Sigma}\wedge
	\omega_{FS}+\int ch_{2}(E)\]
	\[\implies 0<k(\tau-2)\int\omega_{\Sigma}\wedge\omega_{FS}+(\alpha-1+2r_{1}-2r_{2})\int\omega_{\Sigma}\wedge\omega_{FS}\]
	\begin{gather}
		\Rightarrow k(\tau-2)+(\alpha-1)+2(r_{1}-r_{2})>0.
		\label{eq:stability}
	\end{gather}
	Next we write the almost Hermitian Einstein equation for this bundle. The equation is
	\[\frac{1}{2}(\frac{i}{2\pi}R_{E}+k\omega I_{E})^{2}+(\frac{i}{2\pi}R_{E}+k\omega I_{E})\wedge(\frac{c_{1}(\Sigma)+c_{1}(\mathbb{C}\mathbb{P}^{1})}{2})+\frac{c_{1}(\Sigma)\wedge c_{1}(\mathbb{C}\mathbb{P}^{1})}{4}\]
	\[=\frac{\tau}{2}\frac{\chi(X,E\otimes L'^{k})}{Vol(X)}\omega_{\Sigma}\wedge\omega_{FS} I_{E}\]
	
	\[\implies (\frac{i}{2\pi}R_{E}+k\omega I_{E}+\frac{c_{1}(\Sigma)+c_{1}(\mathbb{C}\mathbb{P}^{1})}{2})^{2}=\frac{\tau}{Vol(X)}\chi(X,E\otimes L'^{k})\omega_{\Sigma}\wedge\omega_{FS} I_{E}\]
	
	\[\implies (\frac{i}{2\pi}R_{E}+k\omega I_{E}+\frac{c_{1}(\Sigma)+c_{1}(\mathbb{C}\mathbb{P}^{1})}{2})^{2}=\frac{\tau}{Vol(x)}(\int\alpha\omega_{\Sigma}\wedge\omega_{FS}+\int 2\alpha k\omega_{\Sigma}\wedge\omega_{FS}+\int k\tau\omega_{\Sigma}\wedge\omega_{FS})\omega_{\Sigma}\wedge\omega_{FS} I_{E}\]
	\[+\frac{\tau}{Vol(X)}(\int k(\tau(2r_{2}+1)+2(2r_{1}+1))\omega_{\Sigma}\wedge\omega_{FS}
	+\int \alpha(2r_{2}+1)\omega_{\Sigma}\wedge\omega_{FS})\omega_{\Sigma}\wedge\omega_{FS} I_{E}\]
	\[+\frac{\tau}{Vol(X)}(\int (2r_{1}+1)\omega_{\Sigma}\wedge\omega_{FS}+\int 2k^{2}\tau \omega_{\Sigma}\wedge\omega_{FS}+\int ch_{2}(E))\omega_{\Sigma}\wedge\omega_{FS} I_{E}.\]
	
	Now using $\int\omega_{\Sigma}\wedge\omega_{FS}=\frac{Vol(X)}{\tau},$ we have
	\begin{equation}
		\label{27}
		\implies  (\frac{i}{2\pi}R_{E}+k\omega I_{E}+\frac{c_{1}(\Sigma)+c_{1}(\mathbb{C}\mathbb{P}^{1})}{2})^{2}=\{\alpha+2\alpha k+k\tau+k(\tau(2r_{2}+1)+2(2r_{1}+1))\}\omega_{\Sigma}\wedge\omega_{FS} I_{E}\end{equation}
	\[+\{\alpha(2r_{2}+1)+2r_{1}+1+2k^{2}\tau+2(r_{1}(r_{2}+1)+r_{2}(r_{1}+1))\}\omega_{\Sigma}\wedge\omega_{FS} I_{E}.\]
	\\
	The term $\frac{i}{2\pi}R_{E}+k\omega I_{E}+\frac{c_{1}(\Sigma)+c_{1}(\mathbb{C}\mathbb{P}^{1})}{2}$ equals $\frac{i}{2\pi}R_{E}+k\omega I_{E}+\frac{\alpha}{2}\omega_{\Sigma}+\omega_{FS}.$\\ But $\frac{i}{2\pi}R_{E}+k\omega I_{E}+\frac{\alpha}{2}\omega_{\Sigma}+\omega_{FS}$ is the curvature of the bundle\\
	$E\otimes(\frac{k\tau}{2}L\otimes 2k \mathcal{O}(1))\otimes (\frac{\alpha}{2}L\otimes \mathcal{O}(1)).$ \\A small calculation shows that $E\otimes(\frac{k\tau}{2}L\otimes 2k \mathcal{O}(1))\otimes (\frac{\alpha}{2}L\otimes \mathcal{O}(1))$ equals	\\
	$\pi_{1}^{*}((R_{1}+1)L)\otimes\pi_{2}^{*}(R_{2}\mathcal{O}(2))\oplus\pi_{1}^{*}(R_{1}L)\otimes\pi_{2}^{*}((R_{2}+1)\mathcal{O}(2)),$
	where $R_{1}=r_{1}+\frac{k\tau+\alpha}{2},R_{2}=r_{2}+k+\frac{1}{2}.$\\
	Now $R_{1}-R_{2}>0\implies k(\tau-2)+(\alpha-1)+2(r_{1}-r_{2})>0.$ \par
	We invoke two theorems of \cite{vamsi2}. Theorem $1.4$ of \cite{vamsi2} is as follows :
	\begin{theorem}
		Let $(L,h_{0})$ be a holomorphic line bundle over a compact Riemann surface $M$ such that its curvature $\Theta_{0}$ defines a K\"ahler form $\omega_{\Sigma}=i\Theta_{0}$ over $M$. Assuming the degree($L$) is equal to $1,$ $r_{1},r_{2}\geq 2$ are integers and $\phi\in H^{0}(M,L)$ which is not identically zero, the following are equivalent.
		\begin{enumerate}
			\item Stability : $r_{1}>r_{2}.$
			\item  Existence : There exists a smooth metric $h$ on $L$ such that the curvature $\Theta_{h}$ of its Chern connection $\nabla_{h}$ satisfies the Monge-Amp\`{e}re vortex equation.
			\begin{equation}
				\label{23}
				i\Theta_{h}=(1-\lvert\phi\rvert_{h}^{2})\frac{\mu\omega_{\Sigma}+i\nabla_{h}^{1,0}\phi\wedge\nabla_{h}^{0,1}\phi^{*}}{(2r_{2}+\lvert\phi\rvert_{h}^{2})(2+2r_{2}-\lvert\phi\rvert_{h}^{2})},
			\end{equation}
			where $\mu=2(r_{2}(r_{1}+1)+r_{1}(r_{2}+1))$ and $\phi^{*}$ is the adjoint of $\phi$ with respect to $h$ when $\phi$ is considered as an endomorphism from the trivial line bundle to $L$. 
		\end{enumerate}
		Moreover, if a solution $h$ to \ref{23} satisfying $\lvert\phi\rvert_{h}^{2}\leq 1$ exists, then it is unique. 
	\end{theorem}
	Theorem $4.2$ of \cite{vamsi2} is :
	\begin{theorem}
		Suppose there is a smooth metric $h$ on $L$ satisfying
		\[\lvert\phi\rvert_{h}^{2}\leq 1,\]
		and solving the following equation.
		\[i\Theta_{h}=(1-\lvert\phi\rvert_{h}^{2})\frac{\xi+i\nabla_{h}^{1,0}\phi\wedge\nabla_{h}^{0,1}\phi^{*}}{(2r_{2}+\lvert\phi\rvert_{h}^{2})(2+2r_{2}-\lvert\phi\rvert_{h}^{2})},\]
		where $\xi>0$ is given $(1,1)$-form on $\Sigma$ satisfying 
		\[\int_{\Sigma}^{}\xi=2(r_{1}(r_{2}+1)+r_{2}(r_{1}+1)).\]
		Then there is a smooth Griffiths positively curved metric $H$ on the vortex bundle $V$ whose curvature $\Theta$ satisfies the vbMA equation:
		$(i\Theta)^{2}=\pi_{1}^{*}\xi\wedge\pi_{2}^{*}\omega_{FS}I.$
	\end{theorem}
	Equation \ref{eq:stability} implies that $R_{1}>R_{2}.$ Using the theorems stated above, we have a solution of \ref{27}. Therefore, there exists an almost Hermitian Einstein metric on $E$. \\
	\indent Uniqueness follows from the aforementioned theorems.
\end{proof}
\section{K\"ahlerness of the symplectic form.}\label{kahlerness}
In \cite{Vamsi3}, Pingali gave the moment map interpretation of the Calabi-Yang-Mills equations.	Let $(M,\omega)$ be an $n$-complex dimensaional compact K\"ahler manifold such that $[\omega]=[c_{1}(\tilde{L},h)]$ for some hermitian holomorphic line bundle $(\tilde{L},h)$ satisfying $\int \omega^{n}=1$. Let $(E,\tilde{h})$ be a Hermitian holomorphic vector bundle of rank $r$. The Calabi-Yang-Mills equations(as given in  \cite{Vamsi3}) are
\begin{equation}
	\label{qcym}
	\begin{split}
		&\sqrt{-1}\Theta_{B}\wedge n\omega_{\phi}^{n-1}=-\lambda'\omega_{\phi}^{n}Id\\
		&\omega_{\phi}^{n}(1+\frac{\alpha'\lambda'^{2}r}{2})-\eta'=\alpha' ch_{2}(B)n(n-1)\omega_{\phi}^{n-2},
	\end{split}
\end{equation} 
where $\lambda'$ is a topological constant, $\eta'$ is an $(n,n)$ form, $B$ is a connection and $\omega_{\phi}=\omega + \sqrt{-1}\partial\bar{\partial}\phi,\phi$ is a function on the manifold.\\
If we put $r=2, n=2$ in \ref{qcym}, we get the following.
\begin{equation}
	\label{vortex form}
	\begin{split}
		&\sqrt{-1}\Theta_{B}\wedge \omega_{\phi}=-\frac{\lambda'}{2}\omega_{\phi}^{2}Id\\
		&\omega_{\phi}^{2}(1+\alpha'\lambda'^{2})-\eta'=2\alpha' ch_{2}(B).
	\end{split}
\end{equation}
If we replace $\lambda'$ by $-2\lambda$, $\alpha'$ by $-\frac{\alpha}{2+4\alpha\lambda^{2}}$ and $\frac{\eta'}{1+4\alpha'\lambda^{2}}$ by $\eta$ in \ref{vortex form}, then it becomes
\begin{equation}
	\label{final form}
	\begin{split}
		&\sqrt{-1}\Theta_{B}\wedge \omega_{\phi}=\lambda\omega_{\phi}^{2}Id\\
		&\omega_{\phi}^{2}+\alpha ch_{2}(B)-\eta=0.
	\end{split}
\end{equation}
Vortex bundle is a rank $2$ vector bundle over the manifold $\Sigma\times\mathbb{P}^{1}$, where $\Sigma$ is a Riemann Surface. The Calabi-Yang-Mills equations (as given in \cite{Vamsi1}) on vortex bundle are 
\begin{equation}
	\label{CYM Representability}
	\begin{split}
		&\sqrt{-1}\Theta_{\alpha}\wedge \Omega_{\alpha}=\lambda\Omega_{\alpha}^{2}Id\\
		&\Omega_{\alpha}^{2}+\alpha ch_{2}(E, H_{\alpha})-\eta=0,
	\end{split}
\end{equation}
where $\Omega_{\alpha}>0$ is a smooth form and $H_{\alpha}$ is a smooth metric on the vortex bundle.\\
Now the equations \ref{final form} resemble like \ref{CYM Representability} .\\
The symplectic form on the infinite dimensional manifold $\mathcal{A}_{\mathcal{E}}^{1,1}\times \mathcal{A}^{1,1}$ is
\begin{equation}
	\label{Symplectic form}
	\begin{split}	
		&\frac{(2\pi)^{n+1}}{(\sqrt{-1})^{n-1}}\Omega_{\alpha'}(a_{E}\oplus a_{\tilde{L}},b_{E}\oplus b_{\tilde{L}})=-N\alpha'\int_{M}tr(a_{E}\wedge b_{E})n\Theta_{\tilde{L}}^{n-1}\\ &-N\alpha'\int_{M}(tr(\Theta_{E}a_{E})b_{\tilde{L}}n(n-1)\Theta_{\tilde{L}}^{n-2}+a_{\tilde{L}}tr(\Theta_{E}b_{E})n(n-1)\Theta_{\tilde{L}}^{n-2})\\
		&-N\alpha'\lambda'\int_{M} n\Theta_{\tilde{L}}^{n-1}tr(a_{E})b_{\tilde{L}}-N\alpha'\lambda' \int_{M} n\Theta_{\tilde{L}}^{n-1} a_{\tilde{L}} tr(b_{E})+N\int_{M} a_{\tilde{L}}\wedge b_{\tilde{L}}\wedge n\Theta_{\tilde{L}}^{n-1}\\
		&+N\left(-\alpha'\int_{M}tr(\Theta_{E}^{2})n{n-1 \choose 2}a_{\tilde{L}}\wedge b_{\tilde{L}}\wedge\Theta_{\tilde{L}}^{n-3}-\lambda'\alpha'\int_{M}{n \choose 2}tr(\Theta_{E})\Theta_{\tilde{L}}^{n-2}a_{\tilde{L}}\wedge b_{\tilde{L}}\right),\\
	\end{split}	
\end{equation}
where $\mathcal{A}_{\mathcal{E}}^{1,1}$ is the space of smooth unitary integrable connections on a vector bundle $E$ and  $\mathcal{A}^{1,1}$ is the space of smooth integrable unitary connections on $\tilde{L}$. The tangent space at  $A \in \mathcal{A}_{\mathcal{E}}^{1,1}$ consists of skew-hermitian endomorphism valued 1-forms whose $(0,1)$ part is $d_{A}^{0,1}$ closed(may also be identified with $d_{A}^{0,1}$ closed endomorphism valued $(0,1)$ forms). The tangent spaces at $A\in \mathcal{A}^{1,1}$
consists of $(0,1)$ forms $a^{0,1}$ satisfying $\bar{\partial}a^{0,1}=0$.\\
For $n=2$, the symplectic form is the following.
\begin{equation}
	\label{Symplectic form for n=2}
	\begin{split}	
		&\frac{(2\pi)^{3}}{(\sqrt{-1})}\Omega_{\alpha'}(a_{E}\oplus a_{\tilde{L}},b_{E}\oplus b_{\tilde{L}})=-2N\alpha'\int_{M}tr(a_{E}\wedge b_{E})\Theta_{\tilde{L}}\\ &-2N\alpha'\int_{M}(tr(\Theta_{E}a_{E})b_{\tilde{L}}+a_{\tilde{L}}tr(\Theta_{E}b_{E}))\\
		&-2N\alpha'\lambda'\int_{M} \Theta_{\tilde{L}}tr(a_{E})b_{\tilde{L}}-2N\alpha'\lambda' \int_{M} \Theta_{\tilde{L}} a_{\tilde{L}} tr(b_{E})\\
		&+N\left(-\lambda'\alpha'\int_{M}tr(\Theta_{E})a_{\tilde{L}}\wedge b_{\tilde{L}}\right)
		+2N\int_{M} a_{\tilde{L}}\wedge b_{\tilde{L}}\wedge \Theta_{\tilde{L}}.
	\end{split}	
\end{equation}
We want to check whether this is K\"ahler or not for the vortex bundle ansatz. In \cite{Vamsi3}, the almost complex structure is mentioned. The elements of $\mathcal{A}_{\mathcal{E}}^{1,1}$(where $E$ is the vortex bundle) is of the form 
\begin{equation}
	a_{E} = 
	\begin{bmatrix}
		(\delta A_{h_{1}})^{0,1} & \delta\beta  \\
		0 & (\delta A_{g_{2}})^{0,1}
	\end{bmatrix}
\end{equation} 
The elements of $\mathcal{A}^{1,1}$ is of the form $a_{\tilde{L}}=\pi_{1}^{*}\xi$, where $\xi$ is a $(0,1)$ form on $\Sigma$. We want to check whether $\sqrt{-1}\Omega_{\alpha'}(a_{E}\oplus a_{\tilde{L}},a_{E}^{\dagger}\oplus \overline{a_{\tilde{L}}})$ is positive or negative.
\begin{equation}
	\label{Symplectic form vortex}
	\begin{split}	
		&(2\pi)^{3}\Omega_{\alpha'}(a_{E}\oplus a_{\tilde{L}},a_{E}^{\dagger}\oplus \overline{a_{\tilde{L}}})=-2N\alpha'\int_{\Sigma \times \mathbb{P}^{1}}tr(a_{E}\wedge a_{E}^{\dagger})\sqrt{-1}\Theta_{\tilde{L}}\\ &-2N\alpha'\int_{\Sigma \times \mathbb{P}^{1}}(tr(\sqrt{-1}\Theta_{E}a_{E})\overline{a_{\tilde{L}}}+a_{\tilde{L}}tr(\sqrt{-1}\Theta_{E}a_{E}^{\dagger}))\\
		&-2N\alpha'\lambda'\int_{\Sigma \times \mathbb{P}^{1}} \sqrt{-1}\Theta_{\tilde{L}}tr(a_{E})\overline{a_{\tilde{L}}}-2N\alpha'\lambda' \int_{\Sigma \times \mathbb{P}^{1}} \sqrt{-1}\Theta_{\tilde{L}} a_{\tilde{L}} tr(a_{E}^{\dagger})\\
		&+N\left(-\lambda'\alpha'\int_{\Sigma \times \mathbb{P}^{1}}tr(\sqrt{-1}\Theta_{E})a_{\tilde{L}}\wedge \overline{a_{\tilde{L}}}\right)
		+2N\int_{\Sigma \times \mathbb{P}^{1}} a_{\tilde{L}}\wedge \overline{a_{\tilde{L}}}\wedge \sqrt{-1}\Theta_{\tilde{L}}.
	\end{split}	
\end{equation}
We follow the calculations in \cite{Vamsi1}. We have $\sqrt{-1}\Theta_{\tilde{L}}=\pi_{1}^{*}\omega_{\Sigma}+\frac{4}{\tau}\pi_{2}^{*}\omega_{FS}$(we will ommit the pullback in the following calculations) and
\begin{equation}
	\Theta_{E} = 
	\begin{bmatrix}
		\Theta_{h_{1}}-\beta\wedge\beta^{*} & \nabla^{(1,0)}\beta  \\
		-\nabla^{(0,1)}\beta^{*} & \Theta_{g_{2}}-\beta^{*}\wedge\beta
	\end{bmatrix}.
\end{equation} 
We now calculate the followings.
\begin{equation}
	\label{tr a a dagger}
	\begin{split}
		&tr(a_{E}\wedge a_{E}^{\dagger})\\
		%&=tr \begin{bmatrix}
		%	(\delta A_{h_{1}})^{0,1} & \delta\beta  \\
		%	0 & (\delta A_{g_{2}})^{0,1}
		%\end{bmatrix}
		%\wedge
		%\begin{bmatrix}
		%	\overline{(\delta A_{h_{1}})^{0,1}} & 0  \\
		%	(\delta \beta)^{*} & \overline{(\delta A_{g_{2}})^{0,1}}
	%	\end{bmatrix}\\
		&=(\delta A_{h_{1}})^{0,1}\wedge \overline{(\delta A_{h_{1}})^{0,1}}+\delta\beta \wedge 	(\delta \beta)^{*} + (\delta A_{g_{2}})^{0,1}\wedge \overline{(\delta A_{g_{2}})^{0,1}}
	\end{split}
\end{equation}
\begin{equation}
	\label{tr theta a}
	\begin{split}
		&tr(\Theta_{E}a_{E})\\
	%	&=tr \begin{bmatrix}
	%		\Theta_{h_{1}}-\beta\wedge\beta^{*} & \nabla^{(1,0)}\beta  \\
	%		-\nabla^{(0,1)}\beta^{*} & \Theta_{g_{2}}-\beta^{*}\wedge\beta
	%	\end{bmatrix}
	%	\wedge
	%	\begin{bmatrix}
	%		(\delta A_{h_{1}})^{0,1} & \delta\beta  \\
	%		0 & (\delta A_{g_{2}})^{0,1}
	%	\end{bmatrix}\\
		&=	\Theta_{h_{1}}\wedge 	(\delta A_{h_{1}})^{0,1}-\beta\wedge\beta^{*}\wedge 	(\delta A_{h_{1}})^{0,1}-\nabla^{(0,1)}\beta^{*}\wedge\delta\beta + \Theta_{g_{2}}\wedge(\delta A_{g_{2}})^{0,1}-\beta^{*}\wedge\beta\wedge(\delta A_{g_{2}})^{0,1}
	\end{split}
\end{equation}
\begin{equation}
	\label{tr theta a dagger}
	\begin{split}
		&tr(\Theta_{E}a_{E}^{\dagger})\\
		%&=tr\begin{bmatrix}
		%	\Theta_{h_{1}}-\beta\wedge\beta^{*} & \nabla^{(1,0)}\beta  \\
		%	-\nabla^{(0,1)}\beta^{*} & \Theta_{g_{2}}-\beta^{*}\wedge\beta
		%\end{bmatrix}
		%\wedge
		%\begin{bmatrix}
		%	\overline{(\delta A_{h_{1}})^{0,1}} & 0  \\
		%	(\delta \beta)^{*} & \overline{(\delta A_{g_{2}})^{0,1}}
		%\end{bmatrix}\\
		&=\Theta_{h_{1}}\wedge 	\overline{(\delta A_{h_{1}})^{0,1}}-\beta\wedge\beta^{*}\wedge	\overline{(\delta A_{h_{1}})^{0,1}}+\nabla^{(1,0)}\beta\wedge (\delta \beta)^{*}+ \Theta_{g_{2}}\wedge \overline{(\delta A_{g_{2}})^{0,1}}-\beta^{*}\wedge\beta\wedge\overline{(\delta A_{g_{2}})^{0,1}}
	\end{split}
\end{equation}
%\begin{equation}
%	\label{tr a}
%	tr(a_{E})=	(\delta A_{h_{1}})^{0,1} +(\delta A_{g_{2}})^{0,1}
%\end{equation}
%\begin{equation}
%	\label{tr a dagger}
%	tr(a_{E}^{\dagger})=	\overline{(\delta A_{h_{1}})^{0,1}} +\overline{(\delta A_{g_{2}})^{0,1}}
%\end{equation}
%\begin{equation}
%	\label{tr theta e}
%	tr(\Theta_{E})=\Theta_{h_{1}}-\beta\wedge\beta^{*}+\Theta_{g_{2}}-\beta^{*}\wedge\beta=\Theta_{h_{1}}+\Theta_{g_{2}}
%\end{equation}
We now calculate the terms of \ref{Symplectic form vortex}.\\
%Using \ref{tr a},
 We have 
\begin{equation}
	\label{term 4}
	\begin{split}
		&\sqrt{-1}\Theta_{\tilde{L}}\wedge tr(a_{E})\wedge \overline{a_{\tilde{L}}}
		%&=(\omega_{\Sigma}+\frac{4}{\tau}\omega_{FS})\wedge((\delta A_{h_{1}})^{0,1} +(\delta A_{g_{2}})^{0,1})\wedge \overline{a_{\tilde{L}}}\\
		=-\frac{4}{\tau}\omega_{FS}\wedge\overline{a_{\tilde{L}}}\wedge  (\delta A_{h_{1}})^{0,1}-\frac{4}{\tau}\omega_{FS}\wedge\overline{a_{\tilde{L}}}\wedge  (\delta A_{g_{2}})^{0,1}
	\end{split}
\end{equation}
%Using \ref{tr a dagger}, we have
\begin{equation}
	\label{term 5}
	\begin{split}
		&\sqrt{-1}\Theta_{\tilde{L}}\wedge a_{\tilde{L}}\wedge tr(a_{E}^{\dagger})
	%	&=(\omega_{\Sigma}+\frac{4}{\tau}\omega_{FS})\wedge a_{\tilde{L}}\wedge 	(\overline{(\delta A_{h_{1}})^{0,1}} +\overline{(\delta A_{g_{2}})^{0,1}})\\
		=-\frac{4}{\tau}\omega_{FS}\wedge \overline{(\delta A_{h_{1}})^{0,1}}\wedge a_{\tilde{L}}-\frac{4}{\tau}\omega_{FS}\wedge \overline{(\delta A_{g_{2}})^{0,1}}\wedge a_{\tilde{L}}
	\end{split}
\end{equation}
%Using \ref{tr a a dagger}, we have 
and
\begin{equation}
	\label{term 1}
	\begin{split}
		&tr(a_{E}\wedge a_{E}^{\dagger})\wedge \sqrt{-1}\Theta_{\tilde{L}}\\
		%&=(((\delta A_{h_{1}})^{0,1}\wedge \overline{(\delta A_{h_{1}})^{0,1}}+\delta\beta \wedge 	(\delta \beta)^{*} + (\delta A_{g_{2}})^{0,1}\wedge \overline{(\delta A_{g_{2}})^{0,1}})\wedge(\omega_{\Sigma}+\frac{4}{\tau}\omega_{FS}))\\
		&=-\frac{4}{\tau}\omega_{FS}\wedge\overline{(\delta A_{h_{1}})^{0,1}}\wedge(\delta A_{h_{1}})^{0,1}+\omega_{\Sigma}\wedge\delta\beta\wedge(\delta \beta)^{*}+((\delta A_{g_{2}})^{0,1}\wedge \overline{(\delta A_{g_{2}})^{0,1}})\wedge(\omega_{\Sigma}+\frac{4}{\tau}\omega_{FS})
	\end{split}
\end{equation}
Using \ref{tr theta a} and $\sqrt{-1}\Theta_{g_{2}}=\sqrt{-1}\Theta{f_{2}}+2\omega_{FS}$, we have
\begin{equation}
	\label{term 2}
	\begin{split}
		&tr(\sqrt{-1}\Theta_{E}a_{E})\overline{a_{\tilde{L}}}\\
		%&=\sqrt{-1}(\Theta_{h_{1}}\wedge 	(\delta A_{h_{1}})^{0,1}-\beta\wedge\beta^{*}\wedge 	(\delta A_{h_{1}})^{0,1}-\nabla^{(0,1)}\beta^{*}\wedge\delta\beta + \Theta_{g_{2}}\wedge(\delta A_{g_{2}})^{0,1}-\beta^{*}\wedge\beta\wedge(\delta A_{g_{2}})^{0,1})\wedge\overline{a_{\tilde{L}}}\\
		&=-\sqrt{-1}\beta\wedge\beta^{*}\wedge (\delta A_{h_{1}})^{0,1}\wedge\overline{a_{\tilde{L}}}-\sqrt{-1}\nabla^{0,1}\beta^{*}\wedge\delta\beta\wedge \overline{a_{\tilde{L}}}+2\omega_{FS}\wedge(\delta A_{g_{2}})^{0,1}\wedge\overline{a_{\tilde{L}}}\\
		&-\sqrt{-1}\beta^{*}\wedge\beta\wedge(\delta A_{g_{2}})^{0,1}\wedge\overline{a_{\tilde{L}}}
	\end{split}
\end{equation} 
Using \ref{tr theta a dagger} and
 $\sqrt{-1}\Theta_{g_{2}}=\sqrt{-1}\Theta_{f_{2}}+2\omega_{FS}$, we have 
\begin{equation}
	\label{term 3}
	\begin{split}
		&a_{\tilde{L}}\wedge tr(\sqrt{-1}\Theta_{E} a_{E}^{\dagger})\\
		%&=\sqrt{-1}a_{\tilde{L}}\wedge(\Theta_{h_{1}}\wedge 	\overline{(\delta A_{h_{1}})^{0,1}}-\beta\wedge\beta^{*}\wedge	\overline{(\delta A_{h_{1}})^{0,1}}+\nabla^{(1,0)}\beta\wedge (\delta \beta)^{*}+ \Theta_{g_{2}}\wedge \overline{(\delta A_{g_{2}})^{0,1}}-\beta^{*}\wedge\beta\wedge\overline{(\delta A_{g_{2}})^{0,1}})\\
		&=-\sqrt{-1}\beta\wedge\beta^{*}\wedge a_{\tilde{L}}\wedge\overline{(\delta A_{h_{1}})^{0,1}}+\sqrt{-1}a_{\tilde{L}}\wedge\nabla^{1,0}\beta\wedge (\delta \beta)^{*}+2a_{\tilde{L}}\wedge\omega_{FS}\wedge\overline{(\delta A_{g_{2}})^{0,1}}\\
		&-\sqrt{-1}a_{\tilde{L}}\wedge\beta^{*}\wedge\beta\wedge\overline{(\delta A_{g_{2}})^{0,1}}
	\end{split}
\end{equation}
Using %\ref{tr theta e} and
 $\sqrt{-1}\Theta_{g_{2}}=\sqrt{-1}\Theta_{f_{2}}+2\omega_{FS}$, we have
\begin{equation}
	\label{term 6}
	\begin{split}
		&tr(\sqrt{-1}\Theta_{E})\wedge a_{\tilde{L}}\wedge\overline{a_{\tilde{L}}}
	%	&=\sqrt{-1}(\Theta_{h_{1}}+\Theta_{g_{2}})\wedge a_{\tilde{L}}\wedge\overline{a_{\tilde{L}}}\\
	%	&=\sqrt{-1}\Theta_{g_{2}}\wedge a_{\tilde{L}}\wedge\overline{a_{\tilde{L}}}\\
		=2\omega_{FS}\wedge a_{\tilde{L}}\wedge \overline{a_{\tilde{L}}}
	\end{split}
\end{equation}
If we put \ref{term 1} , \ref{term 2} , \ref{term 3} , \ref{term 4} , \ref{term 5} , \ref{term 6} and $\sqrt{-1}\Theta_{\tilde{L}}=\omega_{\Sigma}+\frac{4}{\tau}\omega_{FS}$ in  \ref{Symplectic form vortex} , then we get 
\begin{equation}
	\label{symp}
	\begin{split}
		&\frac{8N\alpha'}{\tau}\int_{\Sigma \times \mathbb{P}^{1}}\omega_{FS}\wedge\overline{(\delta A_{h_{1}})^{0,1}}\wedge(\delta A_{h_{1}})^{0,1}+2N\alpha'\int_{\Sigma \times \mathbb{P}^{1}}\omega_{\Sigma}\wedge\overline{(\delta A_{g_{2}})^{0,1}}\wedge (\delta A_{g_{2}})^{0,1}\\
		&+\frac{8N\alpha'}{\tau}\int_{\Sigma \times \mathbb{P}^{1}}\omega_{FS}\wedge\overline{(\delta A_{g_{2}})^{0,1}}\wedge (\delta A_{g_{2}})^{0,1}+2N\alpha'\int_{\Sigma \times \mathbb{P}^{1}}\sqrt{-1}\beta\wedge\beta^{*}(-\overline{a_{\tilde{L}}}\wedge(\delta A_{h_{1}})^{0,1}-\overline{(\delta A_{h_{1}})^{0,1}}\wedge a_{\tilde{L}})\\
		&-2N\alpha'\int_{\Sigma \times \mathbb{P}^{1}}\sqrt{-1}\beta\wedge\beta^{*}(-\overline{a_{\tilde{L}}}\wedge(\delta A_{g_{2}})^{0,1}-\overline{(\delta A_{g_{2}})^{0,1}}\wedge a_{\tilde{L}})\\
		&+4N\alpha'\int_{\Sigma \times \mathbb{P}^{1}}\omega_{FS}\wedge(\overline{a_{\tilde{L}}}\wedge(\delta A_{g_{2}})^{0,1}+\overline{(\delta A_{g_{2}})^{0,1}}\wedge a_{\tilde{L}})
		+\frac{8N\alpha'\lambda'}{\tau}\int_{\Sigma \times \mathbb{P}^{1}}\omega_{FS}\wedge(\overline{a_{\tilde{L}}}\wedge(\delta A_{h_{1}})^{0,1}+\overline{(\delta A_{h_{1}})^{0,1}}\wedge a_{\tilde{L}})\\
		&+\frac{8N\alpha'\lambda'}{\tau}\int_{\Sigma \times \mathbb{P}^{1}}\omega_{FS}\wedge(\overline{a_{\tilde{L}}}\wedge(\delta A_{g_{2}})^{0,1}+\overline{(\delta A_{g_{2}})^{0,1}}\wedge a_{\tilde{L}})\\
		&+2N\alpha'\lambda'\int_{\Sigma \times \mathbb{P}^{1}}\omega_{FS}\wedge\overline{a_{\tilde{L}}}\wedge a_{\tilde{L}}-\frac{8N}{\tau}\int_{\Sigma \times \mathbb{P}^{1}}\omega_{FS}\wedge\overline{a_{\tilde{L}}}\wedge a_{\tilde{L}}+2N\alpha'\int_{\Sigma \times \mathbb{P}^{1}}\omega_{\Sigma}\wedge(\delta \beta)^{*}\wedge\delta\beta\\
		&+2N\alpha'\int_{\Sigma \times \mathbb{P}^{1}}\sqrt{-1}(\nabla^{0,1}\beta^{*}\wedge\delta\beta\wedge\overline{a_{\tilde{L}}}+\nabla^{1,0}\beta\wedge(\delta \beta)^{*}\wedge a_{\tilde{L}})
	\end{split}
\end{equation}
Using $\nabla^{0,1}\beta^{*}\wedge\delta\beta\wedge\overline{a_{\tilde{L}}}=\frac{\sqrt{-1}}{\tau}(\delta \phi)\omega_{FS}\wedge\overline{a_{\tilde{L}}}\wedge\nabla^{0,1}\phi^{*}$,\ \  $\nabla^{1,0}\beta\wedge(\delta \beta)^{*}\wedge a_{\tilde{L}}=\frac{\sqrt{-1}}{\tau}\omega_{FS}\wedge\nabla^{1,0}\phi\wedge a_{\tilde{L}}$ and $(\delta\beta)\wedge (\delta \beta)^{*}=\lvert \delta\phi\rvert_{h}^{2}\frac{\sqrt{-1}}{\tau}\omega_{FS}$, \ref{symp} becomes
\begin{equation}
	\begin{split}
		&\int_{\Sigma \times \mathbb{P}^{1}}(\frac{8N\alpha'}{\tau}-\frac{2N\alpha'\lvert\phi\rvert_{h}^{2}}{\tau}-\frac{8N\alpha'\lambda'}{\tau})\omega_{FS}\wedge\overline{(\delta A_{h_{1}})^{0,1}}\wedge(\delta A_{h_{1}})^{0,1}+2N\alpha'\int_{\Sigma \times \mathbb{P}^{1}}\omega_{\Sigma}\wedge\overline{(\delta A_{g_{2}})^{0,1}}\wedge(\delta A_{g_{2}})^{0,1}\\
		&+\int_{\Sigma \times \mathbb{P}^{1}}(\frac{8N\alpha'}{\tau}+\frac{2N\alpha'\lvert\phi\rvert_{h}^{2}}{\tau}+4N\alpha'+\frac{8N\alpha'\lambda'}{\tau})\omega_{FS}\wedge\overline{(\delta A_{g_{2}})^{0,1}}\wedge(\delta A_{g_{2}})^{0,1}\\
		&+\int_{\Sigma \times \mathbb{P}^{1}}(\frac{2N\alpha'\lvert\phi\rvert_{h}^{2}}{\tau}+\frac{8N\alpha'\lambda'}{\tau})\omega_{FS}\wedge\lvert\overline{a_{\tilde{L}}}+\overline{(\delta A_{h_{1}})^{0,1}}\rvert^{2}+\int_{\Sigma \times \mathbb{P}^{1}}(-4N\alpha'-\frac{8N\alpha'\lambda'}{\tau})\omega_{FS}\wedge\lvert\overline{a_{\tilde{L}}}-\overline{(\delta A_{g_{2}})^{0,1}}\rvert^{2}\\
		&-2N\alpha'\int_{\Sigma \times \mathbb{P}^{1}}\lvert\delta\phi\rvert_{h}^{2}\frac{\sqrt{-1}}{\tau}\omega_{FS}\wedge\omega_{\Sigma}+\int_{\Sigma \times \mathbb{P}^{1}}(-\frac{4N\alpha'\lvert\phi\rvert_{h}^{2}}{\tau}+4N\alpha'+2N\alpha'\lambda'-\frac{8N}{\tau})\omega_{FS}\wedge\overline{a_{\tilde{L}}}\wedge a_{\tilde{L}}\\
		&+\frac{2N\alpha'}{\tau}\int_{\Sigma \times \mathbb{P}^{1}}\omega_{FS}\wedge(\lvert\frac{1}{K}\overline{a_{\tilde{L}}}-K(\overline{\delta \phi})\nabla^{1,0}\phi\rvert^{2}-\frac{1}{K^{2}}\overline{a_{\tilde{L}}}\wedge a_{\tilde{L}}-K^{2}\lvert\delta\phi\rvert_{h}^{2}\nabla^{1,0}\phi\wedge\nabla^{0,1}\phi^{*})\\
		&+\int_{\Sigma \times \mathbb{P}^{1}}\frac{2N\alpha'\lvert\phi\rvert_{h}^{2}}{\tau}\omega_{FS}\wedge\lvert\overline{a_{\tilde{L}}}-\overline{(\delta A_{g_{2}})^{0,1}}\rvert^{2},
	\end{split}
\end{equation}
where $K$ is a constant to be chosen later. So the symplectic form is the following.
\begin{equation}
	\label{symplectic with positive term}
	\begin{split}
		&\sqrt{-1}(2\pi)^{3}\Omega_{\alpha'}(a_{E}\oplus a_{\tilde{L}},a_{E}^{\dagger}\oplus \overline{a_{\tilde{L}}})\\
		&=\int_{\Sigma \times \mathbb{P}^{1}}(\frac{8N\alpha'}{\tau}-\frac{2N\alpha'\lvert\phi\rvert_{h}^{2}}{\tau}-\frac{8N\alpha'\lambda'}{\tau})\omega_{FS}\wedge(\sqrt{-1})\overline{(\delta A_{h_{1}})^{0,1}}\wedge(\delta A_{h_{1}})^{0,1}\\
		&+\int_{\Sigma \times \mathbb{P}^{1}}(\frac{8N\alpha'}{\tau}+\frac{2N\alpha'\lvert\phi\rvert_{h}^{2}}{\tau}+4N\alpha'+\frac{8N\alpha'\lambda'}{\tau})\omega_{FS}\wedge(\sqrt{-1})\overline{(\delta A_{g_{2}})^{0,1}}\wedge(\delta A_{g_{2}})^{0,1}\\
		&+\int_{\Sigma \times \mathbb{P}^{1}}(\frac{2N\alpha'\lvert\phi\rvert_{h}^{2}}{\tau}+\frac{8N\alpha'\lambda'}{\tau})\omega_{FS}\wedge(\sqrt{-1})\lvert\overline{a_{\tilde{L}}}+\overline{(\delta A_{h_{1}})^{0,1}}\rvert^{2}\\
		&+\int_{\Sigma \times \mathbb{P}^{1}}(-4N\alpha'-\frac{8N\alpha'\lambda'}{\tau})\omega_{FS}\wedge(\sqrt{-1})\lvert\overline{a_{\tilde{L}}}-\overline{(\delta A_{g_{2}})^{0,1}}\rvert^{2}\\
		&+2N\alpha'\int_{\Sigma \times \mathbb{P}^{1}}\lvert\delta\phi\rvert_{h}^{2}\frac{1}{\tau}\omega_{FS}\wedge\omega_{\Sigma}+\int_{\Sigma \times \mathbb{P}^{1}}(-\frac{4N\alpha'\lvert\phi\rvert_{h}^{2}}{\tau}+4N\alpha'+2N\alpha'\lambda'-\frac{8N}{\tau}-\frac{2N\alpha'}{K^{2}\tau})\omega_{FS}\wedge(\sqrt{-1})\overline{a_{\tilde{L}}}\wedge a_{\tilde{L}}\\
		&+\frac{2N\alpha'}{\tau}\int_{\Sigma \times \mathbb{P}^{1}}\omega_{FS}\wedge(\sqrt{-1})\lvert\frac{1}{K}\overline{a_{\tilde{L}}}-K(\overline{\delta \phi})\nabla^{1,0}\phi\rvert^{2}-\frac{2N\alpha'}{\tau}\int_{\Sigma \times \mathbb{P}^{1}}\omega_{FS}\wedge\lvert\delta\phi\rvert_{h}^{2}K^{2}(\sqrt{-1})\nabla^{1,0}\phi\wedge\nabla^{0,1}\phi^{*}\\
		&+\int_{\Sigma \times \mathbb{P}^{1}}\frac{2N\alpha'\lvert\phi\rvert_{h}^{2}}{\tau}\omega_{FS}\wedge(\sqrt{-1})\lvert\overline{a_{\tilde{L}}}-\overline{(\delta A_{g_{2}})^{0,1}}\rvert^{2}+2N\alpha'\int_{\Sigma \times \mathbb{P}^{1}}\omega_{\Sigma}\wedge(\sqrt{-1})\overline{(\delta A_{g_{2}})^{0,1}}\wedge(\delta A_{g_{2}})^{0,1}.
	\end{split}
\end{equation}
Here, we want to remind ourselves that we made the substitution $\lambda'=-2\lambda<0$ and $\alpha'=-\frac{\alpha}{2+4\lambda^{2}\alpha}\leq 0$ to compare the Calabi-Yang-Mills equations.
From \cite{Vamsi1}, we have 
\begin{equation}
	\label{lambda equality}
	\lambda=\frac{\tau}{8}+\frac{c_{1}(L)\pi}{2vol(\Sigma)}
\end{equation}
and
\begin{equation}
	\label{first chern class inequality}
	0<c_{1}(L)<\frac{\tau vol(\Sigma)}{4\pi}.
\end{equation}
Now
\begin{equation}
	\label{positive 1}
	\begin{split}
		&\frac{2N\alpha'\lvert\phi\rvert_{h}^{2}}{\tau}+\frac{8N\alpha'\lambda'}{\tau}\\
		%&=\frac{N\alpha'}{\tau}(2\lvert\phi\rvert_{h}^{2}+8\lambda')\\
		%&=\frac{N\alpha'}{\tau}(2\lvert\phi\rvert_{h}^{2}-16\lambda)\\
		&=\frac{N\alpha'}{\tau}(2\lvert\phi\rvert_{h}^{2}-2\tau-\frac{c_{1}(L)\pi}{2 vol(\Sigma)})\geq 0.
	\end{split}
\end{equation}
In the second line we substituted $\lambda'$ by $-2\lambda$ and used \ref{lambda equality} . The term \ref{positive 1} is non-negative because $\alpha'\leq0$, $\lvert\phi\rvert_{h}^{2}\leq \tau$ and \ref{first chern class inequality}.\\
Also
\begin{equation}
	\label{positive 2}
	\begin{split}
		&-4N\alpha'-\frac{8N\alpha'\lambda'}{\tau}\\
	%	&=N\alpha'(-4-\frac{8\lambda'}{\tau})\\
	%	&=N\alpha'(-4+\frac{16\lambda}{\tau})\\
		&=N\alpha'(-2+\frac{8c_{1}(L)\pi}{\tau vol(\Sigma)})\geq 0
	\end{split}
\end{equation}
In the second line we substituted $\lambda'$ by $-2\lambda$ and used \ref{lambda equality} . The term \ref{positive 2} is non-negative because $\alpha'\leq 0$ and \ref{first chern class inequality}.\\
Now using the inequality $\lvert a\pm b\rvert^{2}\leq 2\lvert a\rvert^{2}+2\lvert b\rvert^{2}$, \ref{symplectic with positive term} becomes
\begin{equation}
	\label{final expression}
	\begin{split}
			&\sqrt{-1}(2\pi)^{3}\Omega_{\alpha'}(a_{E}\oplus a_{\tilde{L}},a_{E}^{\dagger}\oplus \overline{a_{\tilde{L}}})\\
		&\leq\int_{\Sigma \times \mathbb{P}^{1}}(\frac{8N\alpha'}{\tau}+\frac{2N\alpha'\lvert\phi\rvert_{h}^{2}}{\tau}+\frac{8N\alpha'\lambda'}{\tau})\omega_{FS}\wedge(\sqrt{-1})\overline{(\delta A_{h_{1}})^{0,1}}\wedge(\delta A_{h_{1}})^{0,1}\\
		&+\int_{\Sigma \times \mathbb{P}^{1}}(\frac{8N\alpha'}{\tau}+\frac{2N\alpha'\lvert\phi\rvert_{h}^{2}}{\tau}-4N\alpha'-\frac{8N\alpha'\lambda'}{\tau})\omega_{FS}\wedge(\sqrt{-1})\overline{(\delta A_{g_{2}})^{0,1}}\wedge(\delta A_{g_{2}})^{0,1}\\
		&+\int_{\Sigma \times \mathbb{P}^{1}}\frac{2N\alpha'\lvert\delta\phi\rvert_{h}^{2}}{\tau}\omega_{FS}\wedge(\omega_{\Sigma}-K^{2}(\sqrt{-1})\nabla^{1,0}\phi\wedge\nabla^{0,1}\phi^{*})+2N\alpha'\int_{\Sigma \times \mathbb{P}^{1}}\omega_{\Sigma}\wedge(\sqrt{-1})\overline{(\delta A_{g_{2}})^{0,1}}\wedge(\delta A_{g_{2}})^{0,1}\\
		&+\int_{\Sigma \times \mathbb{P}^{1}}(-4N\alpha'+2N\alpha'\lambda'-\frac{8N}{\tau}-\frac{2N\alpha'}{K^{2}\tau})\omega_{FS}\wedge(\sqrt{-1})\overline{a_{\tilde{L}}}\wedge a_{\tilde{L}}\\
		&+\frac{2N\alpha'}{\tau}\int_{\Sigma \times \mathbb{P}^{1}}\omega_{FS}\wedge(\sqrt{-1})\lvert\frac{1}{K}\overline{a_{\tilde{L}}}-K(\overline{\delta \phi})\nabla^{1,0}\phi\rvert^{2}+\int_{\Sigma \times \mathbb{P}^{1}}\frac{2N\alpha'\lvert\phi\rvert_{h}^{2}}{\tau}\omega_{FS}\wedge(\sqrt{-1})\lvert\overline{a_{\tilde{L}}}-\overline{(\delta A_{g_{2}})^{0,1}}\rvert^{2}.
	\end{split}
\end{equation}
Now
\begin{equation}
	\label{negative 1}
	\begin{split}
		&\frac{8N\alpha'}{\tau}+\frac{2N\alpha'\lvert\phi\rvert_{h}^{2}}{\tau}+\frac{8N\alpha'\lambda'}{\tau}\\
	%	&=\frac{N\alpha'}{\tau}(8+2\lvert\phi\rvert_{h}^{2}+8\lambda')\\
	%	&=\frac{N\alpha'}{\tau}(8+2\lvert\phi\rvert_{h}^{2}-16\lambda)\\
	%	&=\frac{N\alpha'}{\tau}(8+2\lvert\phi\rvert_{h}^{2}-2\tau-\frac{8c_{1}(L)\pi}{2 vol(\Sigma)})\\
		&=\frac{N\alpha}{\tau(2+4\lambda^{2}\alpha)}(-8-2\lvert\phi\rvert_{h}^{2}+2\tau+\frac{8c_{1}(L)\pi}{2 vol(\Sigma)})\\
		&\leq\frac{N\alpha}{\tau(2+4\lambda^{2}\alpha)}(-8+2\tau+\tau)
		<0.
	\end{split}
\end{equation}
In the second line we substituted $\lambda'$ by $-2\lambda$, used \ref{lambda equality}, substituted $\alpha'$ by $\frac{-\alpha}{2+4\lambda^{2}\alpha}$ and in the last line, we used \ref{first chern class inequality}. The last line holds whenever $\tau<\frac{8}{3}$.\\
Also
\begin{equation}
	\label{negative 2}
	\begin{split}
		&\frac{8N\alpha'}{\tau}+\frac{2N\alpha'\lvert\phi\rvert_{h}^{2}}{\tau}-4N\alpha'-\frac{8N\alpha'\lambda'}{\tau}\\
		%&=N\alpha'(\frac{8}{\tau}+\frac{2\lvert\phi\rvert_{h}^{2}}{\tau}-4-\frac{8\lambda'}{\tau})\\
		%&=N\alpha'(\frac{8}{\tau}+\frac{2\lvert\phi\rvert_{h}^{2}}{\tau}-4+\frac{16\lambda}{\tau})\\
	%	&=N\alpha'(\frac{8}{\tau}+\frac{2\lvert\phi\rvert_{h}^{2}}{\tau}-4+2+\frac{8c_{1}(L)\pi}{\tau vol(\Sigma)})\\
		&=\frac{N\alpha}{2+4\lambda^{2}\alpha}(-\frac{8}{\tau}-\frac{2\lvert\phi\rvert_{h}^{2}}{\tau}+2-\frac{8c_{1}(L)\pi}{\tau vol(\Sigma)})\\
		&<	\frac{N\alpha}{2+4\lambda^{2}\alpha}(-\frac{8}{\tau}+2)<0.			
	\end{split}
\end{equation}
In the second line, we substituted $\lambda'$ by $-2\lambda$, used \ref{lambda equality}, substituted $\alpha'$ by $-\frac{\alpha}{2+4\lambda^{2}\alpha}$. In the last line, we used \ref{first chern class inequality}. The last inequality holds whenever $\tau < 4$.\\
First note that $\frac{\tau}{8}<\lambda<\frac{\tau}{4}$ because of \ref{first chern class inequality}.\\
Also  
\begin{equation}
	\label{negative 3}
	\begin{split}
		&-4N\alpha'+2N\alpha'\lambda'-\frac{8N}{\tau}-\frac{2N\alpha'}{K^{2}\tau}\\
		%&=N(-4\alpha'+2\alpha'\lambda'-\frac{8}{\tau}-\frac{2\alpha'}{K^{2}\tau})\\
		%&=N(-4\alpha'-4\alpha'\lambda-\frac{8}{\tau}-\frac{2\alpha'}{K^{2}\tau})\\
		&=N(-4\alpha'-\frac{\tau\alpha'}{2}-\frac{2\alpha' c_{1}(L)\pi}{vol(\Sigma)}-\frac{8}{\tau}-\frac{2\alpha'}{K^{2}\tau})\\
		&<N(-4\alpha'-\frac{\tau\alpha'}{2}-\frac{\tau\alpha'}{2}-\frac{8}{\tau}-\frac{2\alpha'}{K^{2}\tau})\\
		%&=N(-4\alpha'-\tau\alpha'-\frac{8}{\tau}-\frac{2\alpha'}{K^{2}\tau})\\
		%&=N(-\frac{8}{\tau}-\alpha'(4+\tau+\frac{2}{K^{2}\tau}))\\
	%	&=\frac{N}{\tau}(-8-\alpha'(4\tau+\tau^{2}+\frac{2}{K^{2}}))\\
	%&=\frac{N}{\tau}(-8+\frac{\alpha}{2+4\lambda^{2}\alpha}(4\tau+\tau^{2}+\frac{2}{K^{2}}))\\
		&=\frac{N}{\tau(2+4\lambda^{2}\alpha)}(-16-32\lambda^{2}\alpha+\alpha(4\tau+\tau^{2}+\frac{2}{K^{2}}))\\
		&<\frac{N}{\tau(2+4\lambda^{2}\alpha)}(-16-\frac{\tau^{2}\alpha}{2}+4\lambda\alpha+\tau^{2}\alpha+\frac{2\alpha}{K^{2}})\\
		&=\frac{N}{2\tau(2+4\lambda^{2}\alpha)}(-32+8\lambda\alpha+\tau^{2}\alpha+\frac{4\alpha}{K^{2}})
		<0.
	\end{split}
\end{equation}
In the second line, we used $\lambda'=-2\lambda$ and \ref{lambda equality} . In the third line, we used \ref{first chern class inequality} . In the fourth line, we used $\alpha'=-\frac{\alpha}{2+4\lambda^{2}\alpha}$. In the fifth line, we used $\frac{\tau}{8}<\lambda$. The last line holds whenever $-32+8\lambda\alpha+\tau^{2}\alpha+\frac{4\alpha}{K^{2}}<0$ .\\
Let $\xi=(e^{(-s\lvert\phi\rvert_{h}^{2})}\frac{\sqrt{-1}\nabla^{1,0}\phi\wedge\nabla^{0,1}\phi^{*}}{f})$. Let us take
\begin{equation}
	\label{defn w}
	w=\frac{\sqrt{-1}\nabla^{1,0}\phi\wedge\nabla^{0,1}\phi^{*}}{f}.
\end{equation}
Then
\begin{equation}
	\label{del del bar xi}
	\begin{split}
		&\partial\bar{\partial}\xi\\%=\partial[-se^{(-s\lvert\phi\rvert_{h}^{2})}\bar{\partial}\lvert\phi\rvert_{h}^{2}w+e^{(-s\lvert\phi\rvert_{h}^{2})}\bar{\partial}w]\\
		&=s^{2}e^{(-s\lvert\phi\rvert_{h}^{2})}\partial\lvert\phi\rvert_{h}^{2}\wedge\bar{\partial}\lvert\phi\rvert_{h}^{2}w-se^{(-s\lvert\phi\rvert_{h}^{2})}\partial\bar{\partial}\lvert\phi\rvert_{h}^{2}w-se^{(-s\lvert\phi\rvert_{h}^{2})}\partial w\wedge\bar{\partial}\lvert\phi\rvert_{h}^{2}\\
		&-se^{(-s\lvert\phi\rvert_{h}^{2})}\partial\lvert\phi\rvert_{h}^{2}\wedge\bar{\partial}w+e^{(-s\lvert\phi\rvert_{h}^{2})}\partial\bar{\partial}w
	\end{split}
\end{equation}
Suppose $\xi$ achieves its maximum value at a point $q$. Then $\sqrt{-1}\partial\bar{\partial}\xi(q)\leq 0$ and $\partial \xi (q)=0=\bar{\partial}\xi(q)$. We suppress the dependence on $q$, now onwards.
So
\begin{equation}
	\label{del w}
	\begin{split}
		&-se^{(-s\lvert\phi\rvert_{h}^{2})}\partial\lvert\phi\rvert_{h}^{2}w+e^{(-s\lvert\phi\rvert_{h}^{2})}\partial w=0\\
		&\implies \partial w= sw \partial \lvert\phi\rvert_{h}^{2}.
	\end{split}
\end{equation}
Similarly, we have-
\begin{equation}
	\label{del  bar w}
	\bar{\partial} w= sw \bar{\partial} \lvert\phi\rvert_{h}^{2}.
\end{equation}
Now
\begin{equation}
	\label{del del bar xi 1}
	\begin{split}
		&0\geq \sqrt{-1}\partial\bar{\partial}\xi\\
	%	&\implies 0 \geq s^{2}e^{(-s\lvert\phi\rvert_{h}^{2})}\sqrt{-1}\partial\lvert\phi\rvert_{h}^{2}\wedge\bar{\partial}\lvert\phi\rvert_{h}^{2}w-se^{(-s\lvert\phi\rvert_{h}^{2})}\sqrt{-1}\partial\bar{\partial}\lvert\phi\rvert_{h}^{2}w-se^{(-s\lvert\phi\rvert_{h}^{2})}\sqrt{-1}\partial w\wedge\bar{\partial}\lvert\phi\rvert_{h}^{2}\\
%		&-se^{(-s\lvert\phi\rvert_{h}^{2})}\sqrt{-1}\partial\lvert\phi\rvert_{h}^{2}\wedge\bar{\partial}w+e^{(-s\lvert\phi\rvert_{h}^{2})}\sqrt{-1}\partial\bar{\partial}w\\
		&\implies 0 \geq -s^{2}we^{(-s\lvert\phi\rvert_{h}^{2})}\sqrt{-1}\partial\lvert\phi\rvert_{h}^{2}\wedge\bar{\partial}\lvert\phi\rvert_{h}^{2}-swe^{(-s\lvert\phi\rvert_{h}^{2})}\sqrt{-1}\partial\bar{\partial}\lvert\phi\rvert_{h}^{2}+e^{(-s\lvert\phi\rvert_{h}^{2})}\sqrt{-1}\partial\bar{\partial}w\\
		%&\implies 0 \geq -s^{2}w\lvert\phi\rvert_{h}^{2}\sqrt{-1}\nabla^{1,0}\phi\wedge\nabla^{0,1}\phi^{*}-sw\sqrt{-1}(-\Theta_{h}\lvert\phi\rvert_{h}^{2}+\nabla^{1,0}\phi\wedge\nabla^{0,1}\phi^{*})+\sqrt{-1}\partial\bar{\partial}w\\
		&\implies 0 \geq -s^{2}w^{2}\lvert\phi\rvert_{h}^{2}f+sw \lvert\phi\rvert_{h}^{2}f(\frac{\tau-\lvert\phi\rvert_{h}^{2}}{2})(\frac{4}{I}+
		\frac{\alpha w}{2(2\pi)^{2}I})-sw^{2}f+\sqrt{-1}\partial\bar{\partial}w\\
		&\implies 0 \geq f(-s^{2}\lvert\phi\rvert_{h}^{2}+s\lvert\phi\rvert_{h}^{2}(\frac{\tau-\lvert\phi\rvert_{h}^{2}}{2})\frac{\alpha}{2(2\pi)^{2}I}-s)w^{2}+\frac{2s\lvert\phi\rvert_{h}^{2}f(\tau-\lvert\phi\rvert_{h}^{2})}{I}w+\sqrt{-1}\partial\bar{\partial}w
	\end{split}
\end{equation}
In the second line, we used \ref{del w} and \ref{del  bar w}. In the third line, we used $\partial\lvert\phi\rvert_{h}^{2}\wedge\bar{\partial}\lvert\phi\rvert_{h}^{2}=\lvert\phi\rvert_{h}^{2}\nabla^{1,0}\phi\wedge\nabla^{0,1}\phi^{*}$ , $\partial\bar{\partial}\lvert\phi\rvert_{h}^{2}=-\Theta_{h}\lvert\phi\rvert_{h}^{2}+\nabla^{1,0}\phi\wedge\nabla^{0,1}\phi^{*}$ , \ref{defn w} and \ref{22} , where\\ $I=4+\frac{\tau\alpha}{(2\pi)^{2}}(2\lambda-\frac{\tau}{2})+\frac{\tau\alpha}{2(2\pi)^{2}}\lvert\phi\rvert_{h}^{2}-\frac{\alpha}{4(2\pi)^{2}}\lvert\phi\rvert_{h}^{4}$.\\
Now to calculate $\sqrt{-1}\partial\bar{\partial}w$, we need to choose good coordinates and trivialisations. We choose $(z,e)$ such that it is normal for $f$ at the point $q$ and normal for $h$ at the point $q$ with the properties that $\frac{\partial^{2}h}{\partial^{2}z^{2}}(q)=0=\frac{\partial^{2}h}{\partial^{2}\bar{z}^{2}}(q)$. Suppose, in this coordinate $f=\tilde{f} dz \wedge d \bar{z}$. The following calculations are done at the point $q$.\\
First, we calculate
\begin{equation}
	\label{special w}
	\begin{split}
		&w=\frac{\sqrt{-1}\nabla^{1,0}\phi\wedge\nabla^{0,1}\phi^{*}}{f}\\
		&=\frac{\sqrt{-1}(d+A)^{1,0}\phi\wedge(d+A)^{0,1}\phi^{*}}{\tilde{f}dz\wedge d\bar{z}}\\
		&=\frac{\frac{\partial \phi}{\partial z} \frac{\bar{\partial}\phi^{*}}{\partial\bar{z}}+A^{1,0}\phi \frac{\bar{\partial}\phi^{*}}{\partial\bar{z}}}{\tilde{f}}
		=\frac{\frac{\partial \phi}{\partial z} \frac{\bar{\partial}\phi^{*}}{\partial\bar{z}}}{\tilde{f}}
	\end{split}
\end{equation} 
$A^{1,0}$ can be taken to be $0$ at the point $q$.
Now, using \ref{special w} , we have
\begin{equation*}
	\begin{split}
		&\partial\bar{\partial}w\\
	%	&=\partial[\frac{1}{\tilde{f}}\big(\frac{\partial \phi}{\partial z}\bar{\partial} (\frac{\bar{\partial}\phi^{*}}{\partial\bar{z}})+\phi \bar{\partial}A^{1,0} \frac{\bar{\partial}\phi^{*}}{\partial\bar{z}}+\phi A^{1,0}\bar{\partial}(\frac{\bar{\partial}\phi^{*}}{\partial\bar{z}})\big)+\big(\frac{\partial \phi}{\partial z} \frac{\bar{\partial}\phi^{*}}{\partial\bar{z}}+A^{1,0}\phi \frac{\bar{\partial}\phi^{*}}{\partial\bar{z}}\big)\bar{\partial}(\frac{1}{\tilde{f}})]\\
	%	&=\frac{1}{\tilde{f}}\partial[\frac{\partial \phi}{\partial z}\bar{\partial} (\frac{\bar{\partial}\phi^{*}}{\partial\bar{z}})+\phi \bar{\partial}A^{1,0} \frac{\bar{\partial}\phi^{*}}{\partial\bar{z}}+\phi A^{1,0}\bar{\partial}(\frac{\bar{\partial}\phi^{*}}{\partial\bar{z}})]+\partial(\frac{1}{\tilde{f}})\wedge[\frac{\partial \phi}{\partial z}\bar{\partial} (\frac{\bar{\partial}\phi^{*}}{\partial\bar{z}})+\phi \bar{\partial}A^{1,0} \frac{\bar{\partial}\phi^{*}}{\partial\bar{z}}+\phi A^{1,0}\bar{\partial}(\frac{\bar{\partial}\phi^{*}}{\partial\bar{z}})]\\
		%&+\partial\big(\frac{\partial \phi}{\partial z} \frac{\bar{\partial}\phi^{*}}{\partial\bar{z}}+A^{1,0}\phi \frac{\bar{\partial}\phi^{*}}{\partial\bar{z}}\big)\wedge\bar{\partial}(\frac{1}{\tilde{f}})+\big(\frac{\partial \phi}{\partial z} \frac{\bar{\partial}\phi^{*}}{\partial\bar{z}}+A^{1,0}\phi \frac{\bar{\partial}\phi^{*}}{\partial\bar{z}}\big)\partial\bar{\partial}(\frac{1}{\tilde{f}})\\
		&=\frac{1}{\tilde{f}}\partial[\frac{\partial \phi}{\partial z}\bar{\partial} (\frac{\bar{\partial}\phi^{*}}{\partial\bar{z}})+\phi \bar{\partial}A^{1,0} \frac{\bar{\partial}\phi^{*}}{\partial\bar{z}}+\phi A^{1,0}\bar{\partial}(\frac{\bar{\partial}\phi^{*}}{\partial\bar{z}})]+\big(\frac{\partial \phi}{\partial z} \frac{\bar{\partial}\phi^{*}}{\partial\bar{z}}+A^{1,0}\phi \frac{\bar{\partial}\phi^{*}}{\partial\bar{z}}\big)\partial\bar{\partial}(\frac{1}{\tilde{f}})\\
		&=\frac{1}{\tilde{f}}[\partial(\frac{\partial \phi}{\partial z})\wedge\bar{\partial} (\frac{\bar{\partial}\phi^{*}}{\partial\bar{z}})+\partial\phi\wedge \bar{\partial}A^{1,0} \frac{\bar{\partial}\phi^{*}}{\partial\bar{z}}+\phi\partial\bar{\partial}A^{1,0}\frac{\bar{\partial}\phi^{*}}{\partial \bar{z}}]\\
		&+\frac{1}{\tilde{f}}[A^{1,0}\partial\phi\wedge \bar{\partial}(\frac{\bar{\partial}\phi^{*}}{\partial\bar{z}})+\phi\partial A^{1,0}\wedge \bar{\partial}(\frac{\bar{\partial}\phi^{*}}{\partial \bar{z}})]+\big(\frac{\partial \phi}{\partial z} \frac{\bar{\partial}\phi^{*}}{\partial\bar{z}}+A^{1,0}\phi \frac{\bar{\partial}\phi^{*}}{\partial\bar{z}}\big)\partial\bar{\partial}(\frac{1}{\tilde{f}})\\
		&=\frac{1}{\tilde{f}}[\frac{\partial}{\partial z}(\frac{\partial \phi}{\partial z})\frac{\partial}{\partial \bar{z}} (\frac{\bar{\partial}\phi^{*}}{\partial\bar{z}})dz\wedge d\bar{z}+\frac{\partial \phi}{\partial z} \Theta_{h} \frac{\bar{\partial}\phi^{*}}{\partial\bar{z}}+\phi\partial\Theta_{h}\frac{\bar{\partial}\phi^{*}}{\partial \bar{z}}]\\
		&+\frac{1}{\tilde{f}}[A^{1,0}\partial\phi\wedge \bar{\partial}(\frac{\bar{\partial}\phi^{*}}{\partial\bar{z}})+\phi\partial A^{1,0}\wedge \bar{\partial}(\frac{\bar{\partial}\phi^{*}}{\partial \bar{z}})]+\big(\frac{\partial \phi}{\partial z} \frac{\bar{\partial}\phi^{*}}{\partial\bar{z}}+A^{1,0}\phi \frac{\bar{\partial}\phi^{*}}{\partial\bar{z}}\big)\partial\bar{\partial}(\frac{1}{\tilde{f}})\\
	\end{split}
\end{equation*}
\begin{equation}
		\label{long calculation}
		\begin{split}
	%	&=\frac{1}{\tilde{f}}\frac{\partial}{\partial z}(\frac{\partial \phi}{\partial z})\frac{\partial}{\partial \bar{z}} (\frac{\bar{\partial}\phi^{*}}{\partial\bar{z}})dz\wedge d\bar{z}+\frac{1}{\sqrt{-1}}\frac{\partial \phi}{\partial z} \big(\frac{\tau-\lvert\phi\rvert_{h}^{2}}{2I}f(4+\frac{\alpha}{2(2\pi)^{2}}w)\big) \frac{\bar{\partial}\phi^{*}}{\partial\bar{z}}\\
	%	&+\frac{1}{\sqrt{-1}}\phi f\{\frac{\partial}{dz}(\frac{\tau-\lvert\phi\rvert_{h}^{2}}{2I})(4+\frac{\alpha w}{2(2\pi)^{2}})+(\frac{\tau-\lvert\phi\rvert_{h}^{2}}{2I})(\frac{\alpha sw}{2(2\pi)^{2}})\frac{\partial\lvert\phi\rvert_{h}^{2}}{dz}\}\frac{\bar{\partial}\phi^{*}}{\partial \bar{z}} \\
	%	&+\frac{1}{\tilde{f}}[A^{1,0}\partial\phi\wedge \bar{\partial}(\frac{\bar{\partial}\phi^{*}}{\partial\bar{z}})+\phi\partial A^{1,0}\wedge \bar{\partial}(\frac{\bar{\partial}\phi^{*}}{\partial \bar{z}})]+\big(\frac{\partial \phi}{\partial z} \frac{\bar{\partial}\phi^{*}}{\partial\bar{z}}+A^{1,0}\phi \frac{\bar{\partial}\phi^{*}}{\partial\bar{z}}\big)\partial\bar{\partial}(\frac{1}{\tilde{f}})\\
		&=\frac{1}{\tilde{f}}\frac{\partial}{\partial z}(\frac{\partial \phi}{\partial z})\frac{\partial}{\partial \bar{z}} (\frac{\bar{\partial}\phi^{*}}{\partial\bar{z}})dz\wedge d\bar{z}+\frac{1}{\sqrt{-1}}\frac{\partial \phi}{\partial z} \big(\frac{\tau-\lvert\phi\rvert_{h}^{2}}{2I}f(4+\frac{\alpha}{2(2\pi)^{2}}w)\big) \frac{\bar{\partial}\phi^{*}}{\partial\bar{z}}\\
		&+\frac{1}{\sqrt{-1}}\phi f\{(\frac{-I-(\tau-\lvert\phi\rvert_{h}^{2})^{2}\frac{\alpha}{2(2\pi)^{2}}}{2I^{2}})(4+\frac{\alpha w}{2(2\pi)^{2}})+(\frac{\tau-\lvert\phi\rvert_{h}^{2}}{2I})(\frac{\alpha sw}{2(2\pi)^{2}})\}\frac{\partial\lvert\phi\rvert_{h}^{2}}{dz}\frac{\bar{\partial}\phi^{*}}{\partial \bar{z}} \\
		&+\frac{1}{\tilde{f}}[A^{1,0}\partial\phi\wedge \bar{\partial}(\frac{\bar{\partial}\phi^{*}}{\partial\bar{z}})+\phi\partial A^{1,0}\wedge \bar{\partial}(\frac{\bar{\partial}\phi^{*}}{\partial \bar{z}})]+\big(\frac{\partial \phi}{\partial z} \frac{\bar{\partial}\phi^{*}}{\partial\bar{z}}+A^{1,0}\phi \frac{\bar{\partial}\phi^{*}}{\partial\bar{z}}\big)\partial\bar{\partial}(\frac{1}{\tilde{f}})\\
		&=\frac{1}{\tilde{f}}\frac{\partial}{\partial z}(\frac{\partial \phi}{\partial z})\frac{\partial}{\partial \bar{z}} (\frac{\bar{\partial}\phi^{*}}{\partial\bar{z}})dz\wedge d\bar{z}+\frac{1}{\sqrt{-1}}\tilde{f}w \big(\frac{\tau-\lvert\phi\rvert_{h}^{2}}{2I}f(4+\frac{\alpha}{2(2\pi)^{2}}w)\big) \\
		&+\frac{1}{\sqrt{-1}} f\lvert\phi\rvert_{h}^{2}\tilde{f}w\{(\frac{-I-(\tau-\lvert\phi\rvert_{h}^{2})^{2}\frac{\alpha}{2(2\pi)^{2}}}{2I^{2}})(4+\frac{\alpha w}{2(2\pi)^{2}})+(\frac{\tau-\lvert\phi\rvert_{h}^{2}}{2I})(\frac{\alpha sw}{2(2\pi)^{2}})\} \\
		&+\big(\frac{\partial \phi}{\partial z} \frac{\bar{\partial}\phi^{*}}{\partial\bar{z}}+A^{1,0}\phi \frac{\bar{\partial}\phi^{*}}{\partial\bar{z}}\big)\partial\bar{\partial}(\frac{1}{\tilde{f}})\\
	\end{split}
\end{equation}
$\partial A^{1,0}(q)=0$ because $\frac{\partial^{2}h}{\partial z^{2}}=0$.
Now substituting \ref{long calculation} in \ref{del del bar xi 1}, we get
\begin{equation}
	\label{coefficient of w square} 
	\begin{split}
		&0 \geq f(-s^{2}\lvert\phi\rvert_{h}^{2}+s\lvert\phi\rvert_{h}^{2}(\frac{\tau-\lvert\phi\rvert_{h}^{2}}{2})\frac{\alpha}{2(2\pi)^{2}I}-s)w^{2}+\frac{2s\lvert\phi\rvert_{h}^{2}f(\tau-\lvert\phi\rvert_{h}^{2})}{I}w\\
		&+\frac{\sqrt{-1}}{\tilde{f}}\frac{\partial}{\partial z}(\frac{\partial \phi}{\partial z})\frac{\partial}{\partial \bar{z}} (\frac{\bar{\partial}\phi^{*}}{\partial\bar{z}})dz\wedge d\bar{z}+\tilde{f}w \big(\frac{\tau-\lvert\phi\rvert_{h}^{2}}{2I}f(4+\frac{\alpha}{2(2\pi)^{2}}w)\big) \\
		&+ f\lvert\phi\rvert_{h}^{2}\tilde{f}w\{(\frac{-I-(\tau-\lvert\phi\rvert_{h}^{2})^{2}\frac{\alpha}{2(2\pi)^{2}}}{2I^{2}})(4+\frac{\alpha w}{2(2\pi)^{2}})+(\frac{\tau-\lvert\phi\rvert_{h}^{2}}{2I})(\frac{\alpha sw}{2(2\pi)^{2}})\} 
		+\big(\frac{\partial \phi}{\partial z} \frac{\bar{\partial}\phi^{*}}{\partial\bar{z}}+A^{1,0}\phi \frac{\bar{\partial}\phi^{*}}{\partial\bar{z}}\big)\sqrt{-1}\partial\bar{\partial}(\frac{1}{\tilde{f}})\\
		&\implies 0 \geq \frac{\sqrt{-1}}{\tilde{f}}\frac{\partial}{\partial z}(\frac{\partial \phi}{\partial z})\frac{\partial}{\partial \bar{z}} (\frac{\bar{\partial}\phi^{*}}{\partial\bar{z}})dz\wedge d\bar{z}+\big(\frac{\partial \phi}{\partial z} \frac{\bar{\partial}\phi^{*}}{\partial\bar{z}}+A^{1,0}\phi \frac{\bar{\partial}\phi^{*}}{\partial\bar{z}}\big)\sqrt{-1}\partial\bar{\partial}(\frac{1}{\tilde{f}})\\
		&+f\big(-s^{2}\lvert\phi\rvert_{h}^{2}+s\lvert\phi\rvert_{h}^{2}(\frac{\tau-\lvert\phi\rvert_{h}^{2}}{2})\frac{\alpha}{2(2\pi)^{2}I}-s+\tilde{f}\frac{(\tau-\lvert\phi\rvert_{h}^{2})}{2I}\frac{\alpha}{2(2\pi)^{2}}+\tilde{f}\lvert\phi\rvert_{h}^{2}\frac{(\tau-\lvert\phi\rvert_{h}^{2})}{2I}\frac{\alpha s}{2(2\pi)^{2}}\\
		&-\tilde{f}\lvert\phi\rvert_{h}^{2}\frac{(I+(\tau-\lvert\phi\rvert_{h}^{2})^{2}\frac{\alpha}{2(2\pi)^{2}})}{2I^{2}}\frac{\alpha}{2(2\pi)^{2}}\big)w^{2}\\
		&+f\big(\frac{2s\lvert\phi\rvert_{h}^{2}(\tau-\lvert\phi\rvert_{h}^{2})}{I}+\tilde{f}\frac{2(\tau-\lvert\phi\rvert_{h}^{2})}{I}-\tilde{f}\lvert\phi\rvert_{h}^{2}\frac{2(I+(\tau-\lvert\phi\rvert_{h}^{2})\frac{\alpha}{2(2\pi)^{2}})}{I}\big)w\\
	\end{split}
\end{equation}
We can ignore the term $\big(\frac{\partial \phi}{\partial z} \frac{\bar{\partial}\phi^{*}}{\partial\bar{z}}+A^{1,0}\phi \frac{\bar{\partial}\phi^{*}}{\partial\bar{z}}\big)\sqrt{-1}\partial\bar{\partial}(\frac{1}{\tilde{f}})$ in the above maximum principle, whenever Ricci curvature of $f$ is positive. We can choose $s$ small negative number and $\alpha$ small enough such that the coefficient of $w^{2}$ is positive. This implies that $\xi$ is bounded above and the bound does not depend on $f$. 
So we have
\begin{equation}
	\begin{split}
		&\frac{\sqrt{-1}\nabla^{1,0}\phi\wedge\nabla^{0,1}\phi^{*}}{f}<C_{1}
		\implies \frac{\frac{4}{C_{1}V}\sqrt{-1}\nabla^{1,0}\phi\wedge\nabla^{0,1}\phi^{*}}{\omega_{\Sigma}}<1,
	\end{split}
\end{equation}
where $C_{1}$ is a constant and $V=4+\frac{\tau\alpha}{(2\pi)^{2}}(2\lambda-\frac{\tau}{2})+\frac{\tau^{2}\alpha}{2(2\pi)^{2}}$. Now if we choose $K=\frac{2}{\sqrt{C_{1}V}}$, then we have
\begin{equation}
\label{negative 4} \omega_{\Sigma}-K^{2}\sqrt{-1}\nabla^{1,0}\phi\wedge\nabla^{0,1}\phi^{*}>0
\end{equation}
Using \ref{negative 1} , \ref{negative 2} , \ref{negative 3} , \ref{negative 4} , we see that \ref{final expression} is negative whenever $\tau< \frac{8}{3}$ and $\alpha$ is small enough such that the coefficient of $w^{2}$ is positive in \ref{coefficient of w square} and  $-32+8\lambda\alpha+\tau^{2}\alpha+\frac{4\alpha}{K^{2}}<0$(for $K=\frac{2}{\sqrt{C_{1}V}}$). This implies that $-\sqrt{-1}(2\pi)^{3}\Omega_{\alpha'}(a_{E}\oplus a_{\tilde{L}},a_{E}^{\dagger}\oplus \overline{a_{\tilde{L}}})$ is positive whenever the $\tau$ , $\alpha$ satisfies the above conditions  and Ricci curvature of $f$ is positive.

 Department of Mathematics, Indian Institute of Science, Bangalore, India - $560012$\\
   E-mail address: \textit{kartickghosh@iisc.ac.in}

\begin{thebibliography}{9}
\bibitem{LMO1} L.Alvarez-Consul,M. Garcia-Fernandez, and O. Garcia-Prada. ``Coupled equations for
K\"ahler metrics and Yang-Mills connections." Geom. Top. 17, 2731-2812 (2013).
\bibitem{LMO2} L.Alvarez-Consul, M. Garcia-Fernandez, and O. Garcia-Prada. ``Gravitating vortices,
cosmic strings, and the K\"ahler-Yang-Mills equations." 
Comm. Math. Phys. 351 (2017), 361-385.
\bibitem{LMOV}L.Alvarez-Consul, M. Garcia-Fernandez, O. Garcia-Prada, and V. Pingali. ``Gravitating vortices and the Einstein-Bogomolnyi equations''. Math. Ann. (2020), https://doi.org/10.1007/s00208-020-01964-z
\bibitem{Oscar1} 
O. Gar\'cia-Prada. ``Invariant connections and vortices." Commun.Math. Phys., 156 (1993)
527546.
\bibitem{Leung}
Leung, Naichung Conan.  ``Einstein type metrics and stability on vector bundles." J. Differential Geom. 45 (1997), no. 3, 514--546. 
\bibitem{Takahashi}  Takahashi, Ryosuke. ``$J$-equation on holomorphic vector bundles." 2021. arXiv:2112.00550

	\bibitem{Vamsi1} 
	V. Pingali. ``Representability of Chern-Weil forms." 	Math. Zeit., 288 (1-2) (2018) 629-641.
	\bibitem{vamsi2}
	V. Pingali. ``A vector bundle version of the Monge-Amp\`ere
	Equation."  Adv. Math. 360, 106921 (2020).
	\bibitem{Vamsi3} 
	V. Pingali. ``Quillen metrics and perturbed equations." Lett. Math. Phys. 110 (202), 1861-1875. 
\end{thebibliography}
\end{document}